\documentclass[a4paper,fleqn,11pt]{elsarticle}
\usepackage{amssymb,amsthm,amsmath}    
\usepackage{a4wide}
\usepackage{csquotes}
\usepackage{thmtools} %

\usepackage{graphicx} 
\usepackage{tikz,pgfplots} 
    \pgfplotsset{compat=1.15}
    \usepgfplotslibrary{fillbetween}
    \pgfdeclarelayer{ft}
    \pgfdeclarelayer{bg}
    \pgfsetlayers{bg,main,ft}
\usepackage{hyperref} 
\hypersetup{colorlinks,linkcolor={red!50!black},citecolor={blue!50!black},urlcolor={blue!80!black}}

\usepackage{enumitem}
\usepackage{subcaption}
\usepackage{natbib}
\usepackage{doi}

\usepackage{accents}
\usepackage{etoolbox}
\usepackage{array}
\newcolumntype{C}[1]{>{\centering\arraybackslash $}p{#1}<{$}}

\theoremstyle{plain}
\newtheorem{theorem}{Theorem}[section]
\newtheorem{lemma}[theorem]{Lemma}
\newtheorem{proposition}[theorem]{Proposition}

\newtheorem{claim}{Claim}
\newtheorem{corollary}[theorem]{Corollary}
\newtheorem{openproblem}{Open Problem}

\theoremstyle{definition}
\newtheorem{definition}[theorem]{Definition}
\newtheorem{remark}[theorem]{Remark}
\newtheorem{example}[theorem]{Example}

\newcommand{\N}{\mathbb{N}} 
\newcommand{\R}{\mathbb{R}} 

\newcommand{\alphabet}{\Sigma}
\newcommand{\words}{\alphabet^*} 
\newcommand{\eps}{\varepsilon} 
\newcommand{\infwords}{\alphabet^{\N}}

\DeclareMathOperator{\pref}{pref} 
\DeclareMathOperator{\suff}{suff} 
\DeclareMathOperator{\freqname}{freq}
\newcommand{\freq}[2][{}]{\freqname_{#1}(#2)} 
\newcommand{\supfreq}[2][{}]{\overline{\freqname}_{#1}(#2)} 
\newcommand{\inffreq}[2][{}]{\underline{\freqname}_{#1}(#2)} 

\newcommand{\infw}[1]{\mathbf{#1}}

\newcommand{\lang}[2][{}]{\mathcal{L}_{#1}(#2)}
\newcommand{\ablang}[2][{}]{\mathcal{L}^{\text{ab}}_{#1}(#2)}
\newcommand{\soc}[1]{\Omega(#1)}
\newcommand{\abclsr}[1]{\mathcal{A}(#1)}

\newcommand{\complfunction}[1]{\mathcal{P}_{#1}}
\newcommand{\abcomplfunction}[1]{\complfunction{#1}^{\text{ab}}}

\newcommand{\compl}[2][{}]{\complfunction{#1}(#2)}
\newcommand{\abcompl}[2][{}]{\abcomplfunction{#1}(#2)}

\begin{document}

    \title{Abelian Closures of Infinite Binary Words}
		
		\author[2,3]{Svetlana Puzynina
		\corref{cor1}}
		\ead{s.puzynina@gmail.com}
		
		\author[4]{Markus A. Whiteland\corref{cor1}
		}
		\ead{mawhit@mpi-sws.org}
		
	\address[2]{St. Petersburg State University, Russia}
		\address[3]{Sobolev Institute of Mathematics, Russia}
		\address[4]{Max Planck Institute for Software Systems, Saarland Informatics Campus,
		Saarbr{\"u}cken, Germany}
     
     

\begin{abstract}
Two finite words $u$ and $v$ are called Abelian equivalent 
if each letter occurs equally many times in both $u$ and $v$. The
abelian closure $\abclsr{\infw{x}}$ of (the shift orbit closure of) an infinite word
$\infw{x}$ is the set of infinite words $\infw{y}$ such that, for each
factor $u$ of $\infw{y}$, there exists a factor $v$ of $\infw{x}$
which is abelian
equivalent to $u$. 
The notion of an abelian closure gives a characterization of
Sturmian words: among binary uniformly recurrent words, Sturmian
words are exactly those words for which $\abclsr{\infw{x}}$
equals the shift orbit closure $\soc{\infw{x}}$. 
In this paper we show that, contrary to larger alphabets, the abelian closure of a uniformly recurrent aperiodic binary word which is not Sturmian contains infinitely many minimal subshifts.
\end{abstract}

\maketitle

\section{Introduction}

The abelian equivalence relation has been an active topic of research in the recent decades.
Two finite words $u$ and $v$ are called \emph{abelian equivalent} if, for each letter $a$ of
the underlying alphabet $\Sigma$, the words $u$ and $v$ contain equally many occurrences of
$a$. The notion has been studied in the relation of abelian complexity of infinite
words \cite{DBLP:journals/aam/Blanchet-SadriF14,DBLP:journals/dm/MadillR13,DBLP:journals/jlms/RichommeSZ11,DBLP:journals/jalc/Saarela09}, abelian repetitions and avoidance  \cite{DBLP:journals/siamdm/RaoR18, DBLP:journals/eatcs/ConstantinescuI06,DBLP:conf/icalp/Keranen92,PeltomakiW2020:AvoidingAbelianPowersCyclically,DBLP:journals/ijfcs/CassaigneRSZ11}, other topics \cite{DBLP:journals/tcs/FiciMS17,PUZYNINA2013390,DBLP:conf/cwords/PeltomakiW19,PeltomakiW2020:AllGrowthRatesOfAbelianExponents}; see also \cite{DBLP:conf/cwords/Puzynina19} and references therein.

In this note we consider the so-called \emph{abelian closures} of infinite binary words. This
notion is a fairly recent one, and has thus far been considered only in the works
\cite{HejdaSteinerZamboni15,KarhumakiPW:on_abelian_subshifts,DBLP:conf/cwords/Puzynina19}, where the terms ``abelianization" and ``abelian subshift" were used. The notion is motivated by a notion in
discrete symbolic dynamics, namely, the \emph{shift orbit closure} of a word. For an
infinite word $\infw{x}$ we define the \emph{language} $\lang{\infw{x}}$ of $\infw{x}$ as the
set of finite words occurring as factors in $\infw{x}$. The \emph{shift orbit closure}
of an infinite word $\infw{x}$ can be defined as the set $\soc{\infw{x}}$ comprising those infinite words
$\infw{y}$ for which $\lang{\infw{y}} \subseteq \lang{\infw{x}}$. The shift orbit closure has
a discrete symbolic dynamical definition as well: the set $\infwords$ is a compact metric
space under the product topology induced by the discrete topology on the finite alphabet
$\alphabet$. The set $\soc{\infw{x}}$ then coincides with the closure of the orbit of
$\infw{x}$ under the \emph{shift map} $\sigma$, which is defined by
$\sigma(a_0a_1a_2\cdots) = a_1a_2\cdots$. Now $\soc{\infw{x}}$ is called a minimal subshift if it
contains no proper shift orbit closures. The abelian closure of an infinite word can be seen
as the ``commutative" counterpart of its shift orbit closure. The \emph{abelian closure}
$\abclsr{\infw{x}}$ of $\infw{x}$ is defined as the set of words $\infw{y}$ for which each
factor is abelian equivalent to some factor of $\infw{x}$.

The abelian closures of  infinite words can have diverse structures. Clearly, $\soc{\infw{x}} \subseteq \abclsr{\infw{x}}$ for any word $\infw{x}$.
For some words and families of words, for example, Sturmian words, the equality holds: $\soc{\infw{x}} = \abclsr{ \infw{x}}$. Moreover, the
property $ \soc{\infw{x}} = \abclsr{\infw{x}}$ characterizes Sturmian
words among uniformly recurrent binary words \cite{KarhumakiPW:on_abelian_subshifts}. On the other hand, it is easy to see that the
abelian closure of the Thue--Morse word $\infw{TM}$, defined
as the fixed point (starting with $0$) of the morphism $0\mapsto 01$, $1\mapsto 10$, is
$\{\eps,0,1\}\cdot \{01,10\}^{\mathbb{N}}$ (see, e.g., \cite{KarhumakiPW:on_abelian_subshifts} for a proof.)
So, contrary to Sturmian words, the abelian closure of the Thue--Morse is
huge compared to $\Omega_{\infw{TM}}$: essentially, it is a morphic
image of the full binary shift. In general, the abelian closure
of an infinite word might have a pretty complicated structure. T.~Hejda,
W.~Steiner, and L.Q.~Zamboni studied the abelian closure of
the Tribonacci word $\infw{TR}$. They announced that $\soc{\infw{TM}}$ is
a proper subset of $\abclsr{\infw{TM}}$ but that $\soc{\infw{TR}}$ is the only minimal
subshift contained in $\abclsr{\infw{TR}}$ \cite{HejdaSteinerZamboni15,ZamboniPersonal}.

In this paper we consider the abelian closures of binary words. Our main result states that
for an aperiodic uniformly recurrent binary word, its abelian closure contains infinitely many minimal subshifts, unless it is Sturmian (\autoref{thm:binary}). In many cases we are able to
prove that the abelian closure actually contains uncountably many minimal subshifts. We remark that in the 
non-binary case, there exist words with finitely many (and more than one) minimal
subshifts; for example, some balanced aperiodic words are like that (announced in
\cite{KarhumakiPW:on_abelian_subshifts}, see also \autoref{ex:balanced}).

The paper is structured as follows. In \autoref{sec:results} in we give some background and
state our main results.
In
\autoref{sec:preliminaries} we give more technical
prelimininaries we use in the proofs. In particular, we discuss initial properties of abelian
closures of binary words and give some background on Sturmian words. In \autoref{sec:rational}
we prove \autoref{thm:binary} for the easy cases of words which do not have uniform letter
frequencies or which have rational letter frequencies. These cases have been reported at the
DLT 2018 conference \cite{KarhumakiPW:on_abelian_subshifts}, but we give full proofs here for
the sake of completeness. We then prove the theorem for
$C$-balanced words with irrational letter frequencies in \autoref{sec:C-bal}. In \autoref{sec:tools} we develop some
tools we use for the proof of the last and the hardest case of non-balanced words with
irrational frequency, which we treat in \autoref{sec:non-bal}. In \autoref{sec:alt} we give
alternative proofs for some results. For example, we give a large family of words that have
uncountably many minimal subshifts in their abelian closures. In \autoref{sec:conclusions} we conclude with some open problems.

\section{Background and statement of main result}\label{sec:results}
In this section we give some preliminaries on abelian closures and state our main results.

For a finite word $u \in \words$, we let $|u|_a$ denote the number of occurrences of the letter
$a\in\Sigma$ in $u$. A \emph{factor} of a finite or an infinite word is any finite sequence of its consecutive letters. 
The \emph{Parikh vector} $\Psi(u)$ of a finite word $u\in \words$
is defined as $\Psi(x) = (|x|_a)_{a \in \alphabet}$.
The words $u$ and $v$ are \emph{abelian equivalent}, denoted by $u \sim v$, if 
their Parikh vectors coincide.
We let $\lang{\infw{x}}$ denote the language of factors of  an infinite word $\infw{x}$. We then call the set $\ablang[]{\infw{x}} = \{ \Psi(v) \colon v \in \lang[]{\infw{x}} \}$ its
\emph{abelian language}, and an element of $\ablang[]{\infw{x}}$ is referred to as an
\emph{abelian factor} of $\infw{x}$.
In symbols, the definition of the abelian closure of a word reads as follows.
\begin{definition}
The \emph{abelian closure} of
$\infw{x}\in \infwords$ is defined as
\begin{equation*}
    \abclsr{\infw{x}} = \{ \infw{y} \in \infwords \colon 
\ablang{\infw y} \subseteq \ablang{\infw{x}}\}.
\end{equation*}
\end{definition}
In other words, for any factor $u$ of $\infw{y} \in \abclsr{\infw{x}}$ there is a factor $v$ of $\infw{x}$ for which $u \sim v$.
An infinite word $\infw{x}$ is \emph{ultimately periodic} if we may write $\infw{x} = uv^{\omega}$,
i.e., the prefix $u$ is followed by an infinite repetition of a non-empty word $v$. If $u$ is
empty, then $\infw{x}$ is called \emph{purely periodic}. The word $\infw{x}$ is called \emph{aperiodic} if
it is not ultimately periodic. An infinite word $\infw{x}$ is called \emph{recurrent} if each
factor of $\infw{x}$ occurs infinitely many times in $\infw{x}$. We say that a factor $u$
occurs with bounded gaps if there exists $N \in \N$ such that each factor of length $N$
contains $u$. An infinite word $\infw{x}$ is called \emph{uniformly recurrent} if each factor
occurs with bounded gaps.

The term ``abelian subshift'' used in \cite{KarhumakiPW:on_abelian_subshifts} was motivated
by the symbolic dynamical terminology, which we employ here as well.
A \emph{subshift} $X \subseteq \infwords$, $X\neq \emptyset$, is
a closed set (with respect to the product topology of
$\infwords$) satisfying $\sigma(X)\subseteq X$,%
\footnote{Usually subshifts are defined as sets of \emph{bi-infinite} words,
in which case $\sigma(X) = X$ is required in the definition.}
where $\sigma$
is the \emph{shift operator} defined in the introduction. For a subshift $X\subseteq \infwords$ we let
$\lang{X} = \cup_{\infw{y} \in X} \lang{\infw{y}}$. A subshift
$X \subseteq \infwords$ is called \emph{minimal} if $X$ does not
properly contain any subshifts. Observe that two minimal subshifts
$X$ and $Y$ are either equal or disjoint. Let $\infw{x} \in
\infwords$. We let $\soc{\infw{x}}$ denote the \emph{shift
orbit closure} of $\infw{x}$, which may be defined as the subshift
$\{\infw y\in \infwords \colon \lang{\infw{y}} \subseteq
\lang{\infw{x}}\}$. Thus $\lang{\soc{\infw{x}}} =
\lang{\infw{x}}$ for any word $\infw{x} \in \infwords$. It is
known that $\soc{\infw{x}}$ is minimal if and only if $\infw{x}$ is
uniformly recurrent. For more on topic of subshifts we refer the reader to
\cite{LindMarcus95}. We remark that, for any $\infw{x}\in
\infwords$, the abelian closure $\abclsr{\infw{x}}$ is readily seen to be
a subshift.

Sturmian words can be defined in many equivalent ways; here we make of their characterization via balance.
\begin{definition}
An infinite word $\infw{x}\in\infwords$ is called $C$-\emph{balanced}, where $C$ is some positive integer, if for all $v,v'\in \lang{\infw{x}}$ with $|v| = |v'|$, we have
$||v|_a-|v'|_a|\leq C$ for all $a\in \Sigma$. If $\infw{x}$ is not $C$-balanced for any $C \in \N$, then $\infw{x}$
is called \emph{non-balanced}.

A $1$-balanced word is simply called \emph{balanced}. On the other hand, if $\infw{x}$ is not
$1$-balanced, then we call it \emph{unbalanced}.\footnote{Notice that non-balanced words are
unbalanced, but unbalanced words are not necessarily non-balanced.}
\end{definition}
Periodic and aperiodic Sturmian words can then be defined as recurrent balanced binary words
\cite{Morse10.2307/2371431}. It follows that, for each periodic or aperiodic Sturmian
word $\infw{s}$,
its abelian language $\ablang{\infw{x}}$  contains at most two elements of each length. We give more backgrounds on Sturmian words in Sections \ref{sec:preliminaries} and \ref{subsec:Standard}.

In \cite{KarhumakiPW:on_abelian_subshifts} we showed that Sturmian words can be characterized in terms of abelian closures:

\begin{theorem}[\cite{KarhumakiPW:on_abelian_subshifts}]\label{th:St}
Let $\infw{x}$ be uniformly recurrent binary word. Then $\abclsr{\infw{x}}$ contains exactly one minimal subshift if and only if $\infw{x}$ is periodic or aperiodic Sturmian.
\end{theorem}

In fact, for Sturmian words we have $\abclsr{\infw{x}} = \soc{\infw{x}}$. We also
investigated how the property containing exactly one minimal subshift extends to non-binary
words, and we saw that there are many non-binary words with this property.

\begin{example}
Let $\infw s$ be a Sturmian word  and let
$\varphi: 0\mapsto 02, 1\mapsto 12$. Then $\abclsr{\varphi(\infw s)} = \soc{\varphi(\infw
s)}$ \cite{KarhumakiPW:on_abelian_subshifts}.
\end{example}

We also saw that in the non-binary case, there exist words with abelian closure containing more than one but finitely 
many minimal subshifts.

\begin{example}\label{ex:balanced}
Let
$\infw{f} = abaababaa\cdots$ be the \emph{Fibonacci word} over 
the alphabet $\{a,b\}$ defined as the fixed point of the morphism $\varphi: a\mapsto ab$,
$b\mapsto a$.
Consider the words $\infw{u}_1$ and $\infw{u}_2$ obtained from
$\infw{f}$ by replacing the $n$th occurrence of $a$ by the letter $n \pmod{3}$ (resp.,
$-n \pmod{3})$. So $\infw{u}_1 = 0b12b0b12 \cdots$ and
$\infw{u}_2 = 0b21b0b21\cdots$. The words have distinct factors as $2 1$ cannot occur in
$\infw{u}_1$. On the other hand, we have $\infw{u}_2 \in \abclsr{\infw{u}_1}$.
Indeed, it can be shown that the words $\infw{u}_1$ and $\infw{u}_2$ are balanced
(see \cite{Hubert00}). 
Moreover, for any factor $x$ of $\infw{u}_2$ with
$\Psi(x) = (|x|_0,|x|_1,|x|_2,|x|_b)$, $\ablang{\infw{u}_2}$ contains the elements
$(|x|_2,|x|_0,|x|_1,|x|_b)$ and $(|x|_1,|x|_2,|x|_0,|x|_b)$ (i.e., all the cyclic
permutations of the first three elements):
it can be straightforwardly shown that
$(|\varphi^{n}(a)|_a \pmod{3})_{n=0}^{\infty} = (11202210)^{\omega}$. The claim follows from the observations that
$\varphi^{n+k+2}(a) = \varphi^{n+k+1}(a)\varphi^{n+k}(a)$ and $\varphi^n(a)$ is a
prefix of $\varphi^{n+k}(a)$ for all $n,k\geq 0$.
The facts established above imply that at least two of the first three components of $\Psi(x)$ must be
equal. The same observations apply to factors of $\infw{u}_1$. To conclude, notice now that the $0$s in
$\infw{u}_1$ and $\infw{u}_2$ occur in the same positions. So the factor $x'$ in $\infw{u}_1$
occurring at the same position as $x \in \infw{u}_2$ has $\Psi(x') = (|x|_0, |x|_2,|x|_1,|x|_b)$. By a simple case analysis it can be seen that some cyclic permutation
of the first three elements of $\Psi(x')$ equals $\Psi(x)$.

It can be shown that $\abclsr{\infw{u}_1}$ contains exactly two minimal
subshifts: it is the union of $\soc{\infw{u}_1}$ and $\soc{\infw{u}_2}$. Indeed, identifying
the letters $0$, $1$, and $2$ as $a$ of any word $\infw{y}$ in $\abclsr{\infw{u}_1}$ results
in a Sturmian word that is in $\abclsr{\infw{f}}$, and thus in $\soc{\infw{f}}$ by
\autoref{th:St}. Further, removing all $b$s from $\infw{y}$ results in a word that is
in $\abclsr{(012)^{\omega}} = \soc{(012)^{\omega}} \cup \soc{(021)^{\omega}}$. From these observations, it
is then straightforward to conclude that $\infw{y}$ must be in the shift orbit closure of either $\infw{u}_1$ or $\infw{u}_2$.
\end{example}

As the main result of this paper, we show that contrary to the non-binary case, aperiodic
binary words can only contain either one minimal subshift (in the case of Sturmian words) or
infinitely many minimal subshifts:

\begin{theorem}\label{thm:binary}
Let $\infw{x}$ be a binary, aperiodic, uniformly recurrent word which is not Sturmian. Then
$\abclsr{\infw{x}}$ contains infinitely many minimal subshifts.
\end{theorem}

The proof consists of four parts treated in different ways: if $\infw{x}$ does not have 
uniform letter frequencies, the proof is almost immediate. If it has rational letter
frequencies, then using  \emph{standard words} (certain factors of Sturmian words)
we can show that its abelian closure contains uncountably many infinite subshifts (see
\autoref{prop:rationalFreqUncountable}).
The proof for words with irrational frequencies is harder, and is split into the cases of
$C$-balanced words and non-balanced words.

The proof for words which are $C$-balanced for some constant $C$ is provided in
\autoref{Prop:C-bal}. It is geometric in nature and is based on a so-called ``squeezing 
operation'' on infinite binary words. This operation does not extend the language of abelian factors 
of an infinite word, which allows to find infinitely many minimal subshifts in its abelian 
closure. 

The hardest case turns out to be for non-balanced words with irrational letter frequencies
(\autoref{prop:non-balancedAbelianSubshift}). The proof makes use of an operation similar to
the squeezing operation in the $C$-balanced case. Due to non-balancedness, the analysis is
heavily based on deep properties of Sturmian words and standard factorizations. We discuss 
these tools in \autoref{sec:tools}.

\section{Preliminaries and initial properties of abelian closures} \label{sec:preliminaries}

We recall some notation and basic terminology from the literature of combinatorics on
words. We refer the reader to \cite{lothaire1983combinatorics,MR1905123} for more on the
subject. The set of finite words over an alphabet $\alphabet$ is denoted by $\words$. The
empty word is denoted by $\eps$. We let $|w|$ denote the length of a word $w \in \words$. By
convention, $|\eps| = 0$.
The set of right infinite words is denoted by $\infwords$. We refer to infinite words in
boldface font.
Recall that the language $\lang{\infw{x}}$ of an infinite word $\infw{x} \in \infwords$ is the
set of factors of $\infw{x}$. 
The set of length $n$ factors of $\infw{x}$ is denoted by
$\lang[n]{\infw{x}}$, and the set of factors of length at most $n$ is denoted by
$\lang[\leq n]{\infw{x}}$. We use the same notation for finite words as well.

In this paper we are mainly interested in binary words, and we mainly use the alphabet $\{0,1\}$. For a finite binary word $u$,
the \emph{weight} of $u$ refers to $|u|_1$. A binary word is \emph{heavier} than another if it
has larger weight. Similarly it is called \emph{lighter}, if its weight is smaller. Two binary words of equal length are abelian equivalent if and only if they have
equal weight.

For $\infw{x}\in \infwords$
and $a\in\Sigma$, the limits
\begin{equation*}
    \supfreq[\infw{x}]{a} := \lim_{n\to\infty} \frac{\max_{v\in \lang[n]{\infw{x}}} |v|_a}{n}
    \quad \text{and} \quad
    \inffreq[\infw{x}]{a} := \lim_{n\to\infty}\frac{\min_{v\in \lang[n]{\infw{x}}}|v|_a}{n}
\end{equation*}
exist. Furthermore
\begin{equation*}
\supfreq[\infw{x}]{a} = \inf_{n\in\N} \frac{\max_{v\in \lang[n]{\infw{x}}}|v|_a}{n}
\quad \text{and} \quad
\inffreq[\infw{x}]{a} = \sup_{n\in\N}\frac{\min_{v\in \lang[n]{\infw{x}}}|v|_a}{n}.
\end{equation*}
These facts follow from Fekete's lemma, as $\max_{v\in \lang[n]{\infw{x}}} |v|_a$
(resp., $\min_{v\in \lang[n]{\infw{x}}}|v|_a$) is \emph{subadditive} (resp., \emph{superadditive}) with respect to $n$.
It thus follows that $\max_{v \in \lang[n]{\infw{x}}}|v|_a \geq \supfreq[\infw{x}]{a} n$ and
$\min_{v\in \lang[n]{\infw{x}}}|v|_a\leq \inffreq[\infw{x}]{a} n$ for all $n\in \N$. These facts
are used implicitly throughout the paper.
If $\inffreq[\infw{x}]{a} = \supfreq[\infw{x}]{a}$, we denote the common limit by $\freq[\infw{x}]{a}$ and we say that
\emph{$\infw{x}$ has uniform frequency of $a$}.

A morphism $f$ is a mapping $\words \to \Delta^*$, for alphabets $\alphabet$ and $\Delta$, such
that $f(uv) = f(u)f(v)$ for all words $u,v\in \words$. Notice that $f$ is completely defined by
the images of the letters of $\alphabet$. The morphic images of infinite words are defined in a
natural way. A morphism is called \emph{erasing} if $f(a) = \eps$ for some letter $a$. Otherwise it is called non-erasing. For a morphism $f:\Sigma\to \Delta^*$
and a subshift $X\subseteq \infwords$, we define $\varphi(X) = \cup_{\infw{x}\in X}\soc{\varphi(\infw{x})}$. When applying an erasing morphism to a subshift, we make sure that
no element of $X$ gets mapped to a finite word.

Sturmian words enjoy a plethora of different characterizations, and we shall use several of
them in this note. Unless otherwise stated, the results presented below can be found from the
excellent exposition \cite[\S2]{MR1905123}, to which we refer the reader for more on the
topic.

 The \emph{factor complexity function} $\complfunction{\infw{x}} \colon \N \to \N$ is defined by $\compl[\infw{x}]{n} = \# \lang[n]{\infw{x}}$ for each $n\in \N$. Similarly, we define the \emph{abelian complexity function} $\abcomplfunction{\infw{x}} \colon \N \to \N$ of $\infw{x}$ as $\abcompl[\infw{x}]{n} = \# \ablang[n]{\infw{x}}$.
The most commonly used definition of Sturmian words is given via the factor complexity function.
\begin{definition}
An infinite word $\infw{x}$ is Sturmian if $\compl[\infw{x}]{n} = n+1$ for each $n\in \N$.
\end{definition}
Notice that this definition implies that any Sturmian word is binary and is aperiodic by the
famous Morse--Hedlund theorem (see \autoref{thm:MorseHedlundConsecutive} for a formulation).
We shall also consider so-called \emph{periodic Sturmian words}, which we define
later on. To avoid confusion, we follow the convention that, when referring to Sturmian words,
we mean the aperiodic Sturmian words.

It is known that any Sturmian word is uniformly recurrent. Furthermore, a Sturmian word
$\infw{s}$ has irrational uniform letter frequencies. If $\freq[\infw{s}]{1}=\alpha$, then
$\infw{s}$ is called a Sturmian word of \emph{slope} $\alpha$.

As we mentioned in the previous section, Sturmian words can be equivalently defined via balance, and this characterization of Sturmian words is crucial to our considerations:

\begin{theorem}[{\cite[Thm.~2.1.5]{MR1905123}}]
An infinite binary word $\infw{x}$ is Sturmian if and only if it is balanced and aperiodic.
\end{theorem}

Next we consider the structure of factors of Sturmian words.
We recall the so-called \emph{standard pairs} from \cite[Section~2.2]{MR1905123}.
Define two selfmaps $\Gamma$ and $\Delta$ on $\{0,1\}^*\times \{0,1\}^*$ by
\begin{equation*}
	\Gamma(u,v) = (u,uv),\quad \Delta(u,v) = (vu,v).
\end{equation*}

\begin{definition}
	The set of \emph{standard pairs} is the smallest set of pairs of binary words
	containing the pair $(0,1)$ and which is closed under $\Gamma$ and $\Delta$. A \emph{standard word} is any component of a standard pair.

	A word $w$ is called \emph{central}, if $w01$ (or equivalently $w10$) is a standard word.
\end{definition}
For example, the pairs $\Gamma^n(0,1) = (0,0^n1)$ 
and $\Delta^n(0,1) = (1^n0,1)$ are
standard pairs for any $n\geq 0$. Here $1^{n-1}$ and $0^{n-1}$ are central words. These are the only
standard pairs for which one of the components is a letter. Notice also that for a standard
pair $(u,v)$, either $u$ is a letter or $u$ ends with $10$. Similarly either $v$ is a letter or
$v$ ends with $01$. Recall that for a central word $w$ we have that $w01$ is a standard word that ends with $01$. It follows that $w01$ (resp., $w10$) can be expressed as the product $xy$ (resp., $yx$) for a standard pair $(x,y)$. In fact, such a standard pair is unique (see
\cite[Prop.~2.2.1]{MR1905123}).

\begin{definition}
Let $(a_n)_{n\geq 1}$ be a sequence of integers with $a_1 \geq 0$ and $a_n > 0$ for $n > 1$.
We define a sequence of words $S_{-1} = 1$, $S_0 = 0$, and
$S_{n} = S_{n-1}^{a_n} S_{n-2}$ for $n \geq 1$. The sequence $(a_n)_{n\geq 1}$ is called a \emph{directive sequence} and $(S_n)_{n\geq -1}$ is called a \emph{standard sequence}.
\end{definition}
It can be shown that each element $S_n$ of a standard sequence is a standard word. Conversely,
every standard word occurs in some standard sequence. If $a_1 > 0$, then each of the words
$S_n$, $n\geq 0$ starts with $0$. If $a_1 = 0$, then $S_1 = S_{-1} = 1$ and each of the words
$S_n$, $n\geq 1$, starts with $1$. For $n \geq 1$ we have that $S_{2n+1}$ ends with $01$, while $S_{2n}$ ends with $10$.

A standard sequence $(S_n)_{n\geq 1}$ has the property that
$\lim_{n \to \infty} S_n = \infw{s}$ is a Sturmian word. Such a word is called a
\emph{characteristic Sturmian word}. It is the unique element of $\soc{\infw{s}}$ for which
both $0 \infw{s}$ and $1\infw{s} \in \soc{\infw{s}}$. For each directive sequence
$(a_n)_{n\geq 1}$ there is a unique irrational number $\alpha$, such that the corresponding
characteristic Sturmian word $\infw{s}$ has $\freq[\infw{s}]{1} = \alpha$. Conversely, for any
irrational $\alpha \in (0,1)$ there is a corresponding directive sequence which produces the
characteristic Sturmian word having $\freq[\infw{s}]{1} = \alpha$. 

\begin{example}
The \emph{Fibonacci word} $\infw{f} = 01001010\cdots$ is the characteristic Sturmian word
defined by the directive sequence $(1)_{n = 0}^{\infty}$. The directive sequence $(0,1,1,\ldots)$ gives
the Fibonacci word by exchanging $0$ and $1$. The Fibonacci word is the characteristic Sturmian
word of slope $1/\varphi^2$, where $\varphi$ is the \emph{golden ratio}.
\end{example}

Periodic Sturmian words can be equivalently defined as follows:

\begin{definition}
A word is called \emph{periodic Sturmian} if it is an element of $\soc{S^{\omega}}$ for some
standard word $S$.
\end{definition}
To a periodic Sturmian word we may associate a directive sequence and a standard sequence. The
difference is that the directive sequence is finite (with the final element $\omega$). The
slope of a periodic Sturmian word is of course rational, and any rational number is a slope of
some periodic Sturmian word (see \cite[Prop.~2.2.15]{MR1905123}). If two periodic Sturmian words have the same slope, then they define the same shift orbit closure, similar to their
aperiodic counterparts (this can be inferred from the fact that the standard words $xy$ and
$yx$, for a standard pair $(x,y)$, define periodic Sturmian words that are shifts of each other). Hence, we have that the interval $[0,1]$ coincides with the family of slopes of
periodic and aperiodic Sturmian words.

Periodic and aperiodic Sturmian words are exactly the \emph{recurrent} balanced binary words
\cite{Morse10.2307/2371431}. It follows that, for each periodic or aperiodic Sturmian
word $\infw{s}$, the abelian language $\ablang[n]{\infw{x}}$ consists of at most two elements.
(For aperiodic Sturmian words it is always equal to $2$, as it is easy to see that
a word having $\abcompl[\infw{x}]{n} = 1$ for some $n$ is purely periodic \cite{DBLP:journals/mst/CovenH73}. There
also exist non-recurrent balanced binary words. For us, it suffices to know that for
any standard word $S$, the words $0S^{\omega}$ and $1S^{\omega}$ are balanced.

 We now recall some preliminary observations on
abelian closures of infinite words. The results appear in \cite{KarhumakiPW:on_abelian_subshifts} unless otherwise stated.

\begin{lemma}\label{lem:uniformFrequencies}
Assume $\infw{x}\in\infwords$ has uniform frequency of a letter
$a\in\Sigma$. Then any word $\infw y\in \abclsr{\infw{x}}$ has
uniform frequency of $a$ and $\freq[\infw{y}]{a}= \freq[\infw{x}]{a}$.
\end{lemma}

We immediately have that if $\infw{x}$ has an irrational uniform frequency
of some letter $a$, then $\abclsr{\infw{x}}$ contains only aperiodic words.
We continue by observing how the abelian closures of
periodic and ultimately periodic words can differ.

\begin{proposition}\label{prop:periodicFinite}
For any purely periodic word $\infw{x}$, the abelian closure $\abclsr{\infw{x}}$ is finite.
\end{proposition}

The abelian closure of an ultimately, but not purely periodic word can be huge; in fact, it can 
contain uncountably many minimal subshifts. This was already observed in \cite{HejdaSteinerZamboni15}, and further examples
were given in \cite[Ex.~2]{KarhumakiPW:on_abelian_subshifts}.

We conclude this section by recalling two rather straightforward observations, which will be 
used throughout the paper. The first one is immediate by a "sliding window" argument and is
well-known in the literature. The second one is a straightforward consequence of the first.

\begin{lemma}[Continuity of abelian complexity]\label{lem:continuity}
Let $\infw{u}$ be an infinite binary word and $(s_1,t_1)$ and $(s_2,t_2)$ with $s_1<s_2$ be two
elements of $\ablang[n]{\infw{u}}$. Then each $(s,t)$ with $s+t=n$ and $s_1<s<s_2$ is an element of $\ablang[n]{\infw{u}}$.
\end{lemma}

\begin{lemma}[Corridor Lemma]\label{lem:corridorLemma}
Let $\infw{x}$ be a binary word. Then $\infw y\in \abclsr{\infw{x}}$ if and only if, for
all $n\in\N$,
\begin{align*}
\min\{|v|_1\colon v\in \lang[n]{\infw{y}}\} &\geq \min\{|v|_1\colon v\in \lang[n]{\infw{x}}\} \text{ and}\\
\max\{|v|_1\colon v\in \lang[n]{\infw y}\} &\leq \max\{|v|_1\colon v\in \lang[n]{\infw{x}}\}.
\end{align*}
\end{lemma}

\section{Rational letter frequencies and no letter frequencies}\label{sec:rational}

In this section, we prove easy parts of \autoref{thm:binary}: the case when letter frequencies
do not exist, and the case when they exist and are rational. As mentioned previously, the
results were reported in \cite{KarhumakiPW:on_abelian_subshifts}. We give the full proofs here
for the sake of completeness, and we will further discuss further aspects of them
in \autoref{sec:alt}.

\begin{proposition}\label{prop:NoFrequency}
Let $\infw{x}$ be a binary word having no uniform letter frequencies. Then
$\abclsr{\infw{x}}$ contains uncountably many
minimal subshifts.
\end{proposition}
\begin{proof}
Let $\alpha = \supfreq[\infw{x}]{1} > \inffreq[\infw{x}]{1} =  \alpha'$.
Then, for any Sturmian word $\infw s$ of slope $\beta$, where $\alpha'\leq \beta\leq
\alpha$, $\soc{\infw s}$ is contained in $\abclsr{\infw{x}}$
by the \hyperref[lem:corridorLemma]{Corridor Lemma}. There are
uncountably many such $\infw s$.
\end{proof}

We then turn to uniformly recurrent binary words having rational uniform
letter frequencies. Our aim is to prove the following proposition:

\begin{proposition}\label{prop:rationalFreqUncountable}
Let $\infw{x}\in\{0,1\}^{\N}$ be uniformly recurrent and aperiodic with
rational uniform letter frequencies. Then $\abclsr{\infw{x}}$
contains uncountably many minimal subshifts.
\end{proposition}

For the remainder of the section we fix the word $\infw{x}$ to be uniformly recurrent and aperiodic with $\freq{\infw{x}}(1) = p/q$ and we assume $\gcd(p,q) = 1$.
We begin with a few technical lemmas:
\begin{lemma}\label{lem:AperBinRatFreqInfSup}
For all $n \in \N$ we have
$\max_{v\in \lang[n]{\infw{x}}}|v|_1 > n\frac{p}{q}$
and $\min_{v\in \lang[n]{\infw{x}}}|v|_1 < n\frac{p}{q}$.
\end{lemma}

\begin{proof}
We show the claim for $\max_{v\in \lang[n]{\infw{x}}}|v|_1$. The proof for $\min_{v\in \lang[n]{\infw{x}}}|v|_1 < n\frac{p}{q}$ is symmetric. If $n p/q$ is not an integer, then the claim follows from the fact that
$\max_{v \in \lang[n]{\infw{x}}}|v|_1 \geq n p/q$. For the sake of contradiction,
assume that $\max_{v\in \lang[n]{\infw{x}}}|v|_1 = n\frac{p}{q}$ for some multiple $n$ of $q$.
Write $\infw{x} = a_1 a_2\cdots$. For any $M \geq n$ we
may write
\begin{equation*}
|\infw{x}_{[1,M+n)}|_1 = \frac{1}{n} \sum_{i=1}^{M}|\infw
x_{[i,i+n)}|_1
           + \frac{1}{n}\sum_{i=1}^{n-1}(n-i)(|a_i|_1+|a_{M+n-i}|_1).
\end{equation*}
Indeed, in the first sum each $|a_i|_1$ is counted $n$ times for
$i=n,\ldots,M$. For $i\in
\{1,\ldots,n-1\}$ the values $|a_i|_1$ and $|a_{M+n-i}|_1$ are
counted $i$ times each. Hence, the first sum equals
$|\infw{x}_{[n,M]}|_1 + \frac{1}{n}\sum_{i=1}^{n-1} i(|a_i|_1 + |a_{M + n - i}|_1)$. The
second sum adds the missing contributions so that the total contribution of each letter
is counted once after normalizing by $\frac{1}{n}$. Observe that
the second sum is bounded from above by $n-1$ (after dividing by $\frac{1}{n}$).

As $\infw{x}$ is uniformly recurrent, there exists $N \in \N$ such that each
factor of length $N$ contains a factor of length $n$ having at most
$n\frac{p}{q} - 1$ occurrences of $1$.
Recall that each factor of length $n$ has weight at most
$np/q$ by assumption. Hence for all 
$M\geq 1$
\begin{equation*}
\frac{1}{n}\sum_{i=1}^{MN}|\infw{x}_{[i,i+n)}|_1\leq
M(N\tfrac{p}{q}-\tfrac{1}{n})
\end{equation*}
since at least $M$ of the factors of length $n$ of $\infw{x}_{[1,MN+n)}$ have at most
$n\frac{p}{q}-1$ occurrences of the letter $1$. But now
\begin{align*}
\lim_{M\to\infty}\frac{1}{MN+n}|\infw{x}_{[1,MN+n)}|_1 & \leq \lim_{M\to\infty}\frac{1}{MN+n}\sum_{i=1}^{MN}\frac{1}{n}|\infw{x}_{[i,i+n)}|_1 + \frac{n-1}{MN + n}\\
                                                   &\leq \lim_{M\to\infty}\frac{1}{MN+n}(MN\frac{p}{q}-\frac{M}{n})\\
                                                   & =\frac{p}{q}-\frac{1}{nN}.
\end{align*}
This is a contradiction.
\end{proof}

As $\infw{x}$ is assumed to be aperiodic, an immediate consequence of the above is that for any multiple $n$ of $q$, the values $n\frac{p}{q}-1$ and $n\frac{p}{q}+1$ are the weights of some factors of length $n$ of $\infw{x}$.

Let now $\varphi:\{0,1\} \to \{0,1\}^*$ be defined by $0\mapsto w01$,
$1\mapsto w10$, where $w01$ (or $w10$) is the Standard word of
slope $\frac{p}{q}$ having $|w01|_1 = p$ and $|w01| = q$.

\begin{lemma}\label{lem:rational}
For all $n\in\N$ and $v\in \lang[n]{\varphi(\{0,1\}^{\N})}$ we have $||v|_1 - n p/q| \leq 1$.
Furthermore, each integral value in the interval $[n p/q -1, np/q + 1]$ is the weight of some
$v\in \lang[n]{\varphi(\{0,1\}^{\N}}$.
\end{lemma}
\begin{proof}
Let $n\in \N$ and $u\in \lang[n]{\varphi(\{0,1\}^{\N}}$. Then there exist letters
$a,b,c,d$ with $\{a,b\} = \{c,d\} = \{0,1\}$ such that
$u = r\varphi(x)s$
for some $x\in \{0,1\}^*$, $r\in \suff(wab)$, and $s\in \pref(wcd)$ satisfying $|r|,|s|< q$.
Observe now that $u\sim_{\text{ab}} v$ for some $v\in r(wcd)^*s$ and that $(wcd)^{\omega}$ is periodic Sturmian.

If $|r|\geq 2$, $r=\eps$, or $r=b=d$, then $v\sim u'$ for some
$u'\in \lang{(wab)^{\omega}}$. It follows that
$|u|_1 \in \{\lfloor n\frac{p}{q}\rfloor,\lceil n\frac{p}{q}\rceil\}$. Assume that
$r = b \neq d$. Now $u\sim v = \pref_n(b(wba)^{\omega})$, where
$b(wba)^{\omega}$ is balanced.
If $n$ is not a multiple of $q$, then $|u|_1 = \lfloor
n\frac{p}{q}\rfloor +|b|_1$. If $n$ is a multiple of $q$, then
$|u|_1 = n\frac{p}{q} + |b|_1-|a|_1$.

We have shown that $| |u|_1 - np/q| \leq 1$ regardless of whether $n$ is a multiple of
$q$ or not. Clearly each
value is attained by some word in $\mathcal
L_n(\varphi(\{0,1\}^{\N}))$. This concludes the proof.
\end{proof}

The above lemmas allow us to conclude \autoref{prop:rationalFreqUncountable}.
\begin{proof}[Proof of \autoref{prop:rationalFreqUncountable}]
Assume $\freq[\infw{x}]{1} = \frac{p}{q}$. Let $\varphi(\{0,1\}^{\N}) = \mathcal O$ be as in the above lemma
whence, for all $v\in \lang[n]{\mathcal O}$,
$||v|_1 - np/q| \leq 1$. By
\autoref{lem:AperBinRatFreqInfSup}, $\max_{u\in \lang[n]{\infw{x}}}|u|_1
> n\frac{p}{q}$ and $\min_{u\in \lang[|v|]{\infw{x}}}|u|_1
< n\frac{p}{q}$. By the \hyperref[lem:corridorLemma]{Corridor Lemma}
we have that, for any word $\infw y\in \mathcal O$,
we have $\soc{\infw y} \in \abclsr{\infw{x}}$. Clearly
$\mathcal O$ contains uncountably many minimal
subshifts.
\end{proof}

\section{Abelian closures of \texorpdfstring{$C$}{C}-balanced words}\label{sec:C-bal}

In this section we prove the following statement:

\begin{proposition}\label{Prop:C-bal}
Let $\infw{x}$ be a uniformly recurrent binary word which is not Sturmian. Suppose in addition
that $\infw{x}$ is $C$-balanced for some $C>1$ and that $\freq[\infw{x}]{1} = \alpha$ is
irrational. Then $\abclsr{\infw{x}}$ contains infinitely many minimal subshifts.
\end{proposition}

We use the following notion of the graph $g_{\infw{w}}$ of an infinite word $\infw{w}$, which is a modification
of a geometric approach from \cite{AP16}. We focus on binary words, although the notion extends
in an obvious way to nonbinary alphabets. Let $\infw{w}=a_1 a_2 \cdots$ be an infinite word over
a finite alphabet $\alphabet$. We translate $\infw{w}$ to a graph visiting points of the
infinite rectangular grid by interpreting letters of $w$ as drawing instructions. In the binary
case, we associate the letter $0$ with a move by vector $\vec{v}_0=(1,0)$, and the letter $1$ with a move
$\vec{v}_1=(1,1)$. We start at the origin $(x_0,y_0)=(0,0)$. At step $n$, we are at a point
$(x_{n-1}, y_{n-1})$ and we move by a vector corresponding to the letter $a_{n}$, so that we
come to a point $(x_{n}, y_{n})=(x_{n-1}, y_{n-1}) + \vec{v}_{a_n}$, and the two points
$(x_{n-1}, y_{n-1})$ and $(x_{n}, y_{n})$ are connected with a line segment. So, we translate
the word $\infw{w}$ to a path in $\mathbb{Z}^2$. We denote the corresponding graph by
$g_{\infw{w}}$. So, for any word $\infw{w}$, its graph is a piecewise linear function with
linear segments connecting integer points (see \autoref{fig:example_TM}). We remark that
$g_{\infw{w}}(i)=|a_1\cdots a_{i}|_1$. Note also that instead of the vectors $(0,1)$ and
$(1,1)$, one can use any other pair of noncollinear vectors $\vec{v}_0$ and $\vec{v}_1$. For a
$k$-letter alphabet one can consider a similar graph in $\mathbb{Z}^k$. Note that the graph
can also be defined for finite words in a similar way, and we will sometimes use it.

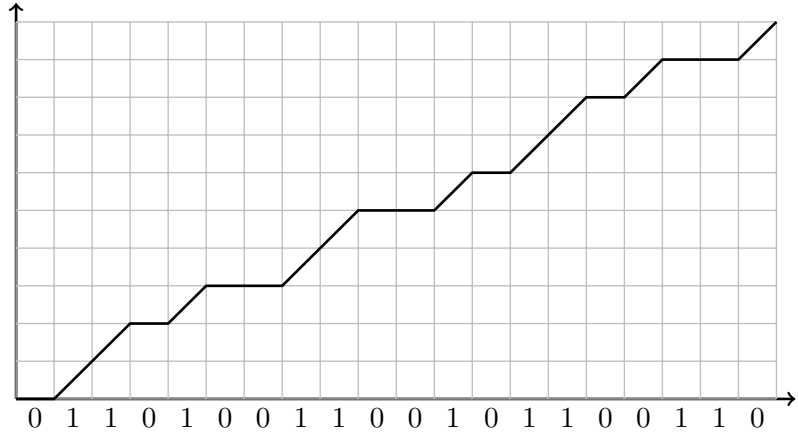
\begin{figure}
\centering

\begin{tikzpicture}[scale=0.5]
\tikzstyle{every node}=[shape=rectangle,fill=none,draw=none,minimum size=0cm,inner sep=2pt]
    \tikzstyle{every path}=[draw=black,line width = 1pt]

    \draw[->] (0,0)  to   (20.5,0);
    \draw[->] (0,0)  to   (0,10.5);

    \tikzstyle{every path}=[draw=black,line width = 0.5pt]    
\draw [gray!50] (0,0) grid (20,10);
    
    \tikzstyle{every path}=[draw=black,line width = 1pt]      
\draw (0,0)  --  (1,0) -- (2,1) -- (3,2) -- (4,2) -- (5,3) -- (6,3) -- (7,3) -- (8,4) -- (9,5) -- (10,5) -- (11,5) -- (12,6) -- (13,6) -- (14,7) -- (15,8) -- (16,8) -- (17,9) -- (18,9) -- (19,9) -- (20,10);

 \node[fill=white] at (0.5,-0.5){$0$};
  \node[fill=white] at (1.5,-0.5){$1$};
   \node[fill=white] at (2.5,-0.5){$1$};
    \node[fill=white] at (3.5,-0.5){$0$};

 \node[fill=white] at (4.5,-0.5){$1$};
  \node[fill=white] at (5.5,-0.5){$0$};
   \node[fill=white] at (6.5,-0.5){$0$};
    \node[fill=white] at (7.5,-0.5){$1$};
    
 \node[fill=white] at (8.5,-0.5){$1$};
  \node[fill=white] at (9.5,-0.5){$0$};
   \node[fill=white] at (10.5,-0.5){$0$};
    \node[fill=white] at (11.5,-0.5){$1$};
    
 \node[fill=white] at (12.5,-0.5){$0$};
  \node[fill=white] at (13.5,-0.5){$1$};
   \node[fill=white] at (14.5,-0.5){$1$};
    \node[fill=white] at (15.5,-0.5){$0$};
    
 \node[fill=white] at (16.5,-0.5){$0$};
  \node[fill=white] at (17.5,-0.5){$1$};
   \node[fill=white] at (18.5,-0.5){$1$};
    \node[fill=white] at (19.5,-0.5){$0$};
    
\end{tikzpicture}
\caption{Graph of the Thue--Morse word.}\label{fig:example_TM}
\end{figure}

To prove the proposition, we will need the following operation of (upper) $C$-squeezing.

\begin{definition}
Let $\infw{x} = a_1a_2\cdots$ be a binary word with $\freq[\infw{x}]{1}  = \alpha$, and let $C\in\mathbb{R}$. We define an operation of \emph{$C$-squeezing} of $\infw{x}$,
$s_C^{+}(\infw{x}) = a_1' a_2' \cdots$, as follows. For each $i$ such that
$g_{\infw{x}}(i) > \alpha i + C$ and $a_{i-1}=1$, $a_i=0$, we define $a'_{i-1}=0$, $a'_i=1$.  In this case we say that we have a switch at position $i$.
Otherwise we define $a'_i=a_i$.
\end{definition}

 Informally, the $C$-squeezing operation works as follows. If there is a piece of graph above
 the line $y=\alpha x + C$, we make local changes in this piece getting this part of the graph
 closer to the stripe (see \autoref{fig:uppersqueezing}). Clearly, we can symmetrically define
 an operation  of lower $C$-squeezing, but for our proof it is enough to squeeze only from one
 side. We return to this operation in \autoref{subsec:C-squeezingRevisited}.

\begin{figure} \centering
\begin{tikzpicture}[scale=0.5]
\tikzstyle{every node}=[shape=rectangle,fill=none,draw=none,minimum size=0cm,inner sep=2pt]
    \tikzstyle{every path}=[draw=black,line width = 1pt]

    \draw[->] (0,0)  to   (20.5,0);
    \draw[->] (0,0)  to   (0,10.5);

    \tikzstyle{every path}=[draw=black,line width = 0.5pt]    
\draw [gray!50] (0,0) grid (20,10);
\draw (0.5,0) -- (20,9.2);
    
    \tikzstyle{every path}=[draw=black,line width = 1pt]      
\draw (0,0)  --  (1,0) -- (2,1) -- (3,2) -- (4,2) -- (5,2) -- (6,2) -- (7,2) -- (8,3) -- (9,3) -- (10,4) -- (11,5) -- (12,6) -- (13,7) -- (14,7) -- (15,8) -- (16,8) -- (17,8) -- (18,8) -- (19,9) -- (20,9);

\draw[dotted] (2,1) -- (3,1) -- (4,2)
(12,6) -- (13,6) -- (14,7) -- (15,7) -- (16,8)
(18,8) -- (19,8) -- (20,9);

 \node[fill=white] at (0.5,-0.5){$0$};
  \node[fill=white] at (1.5,-0.5){$1$};
   \node[fill=white] at (2.5,-0.5){$1$};
    \node[fill=white] at (3.5,-0.5){$0$};

 \node[fill=white] at (4.5,-0.5){$0$};
  \node[fill=white] at (5.5,-0.5){$0$};
   \node[fill=white] at (6.5,-0.5){$0$};
    \node[fill=white] at (7.5,-0.5){$1$};
    
 \node[fill=white] at (8.5,-0.5){$0$};
  \node[fill=white] at (9.5,-0.5){$1$};
   \node[fill=white] at (10.5,-0.5){$1$};
    \node[fill=white] at (11.5,-0.5){$1$};
    
 \node[fill=white] at (12.5,-0.5){$1$};
  \node[fill=white] at (13.5,-0.5){$0$};
   \node[fill=white] at (14.5,-0.5){$1$};
    \node[fill=white] at (15.5,-0.5){$0$};
    
 \node[fill=white] at (16.5,-0.5){$0$};
  \node[fill=white] at (17.5,-0.5){$0$};
   \node[fill=white] at (18.5,-0.5){$1$};
    \node[fill=white] at (19.5,-0.5){$0$};
    
\draw (6,-2)  --  (7,-2);
\draw[densely dotted] (10,-2)  --  (11,-2);

 \node[fill=white] at (7.5,-2){$\infw{w}$};
 \node[fill=white] at (13,-2){$\infw{w}'=s^+_C(\infw{w})$};    
\end{tikzpicture}
\caption{Upper squeezing.}\label{fig:uppersqueezing}
\end{figure}
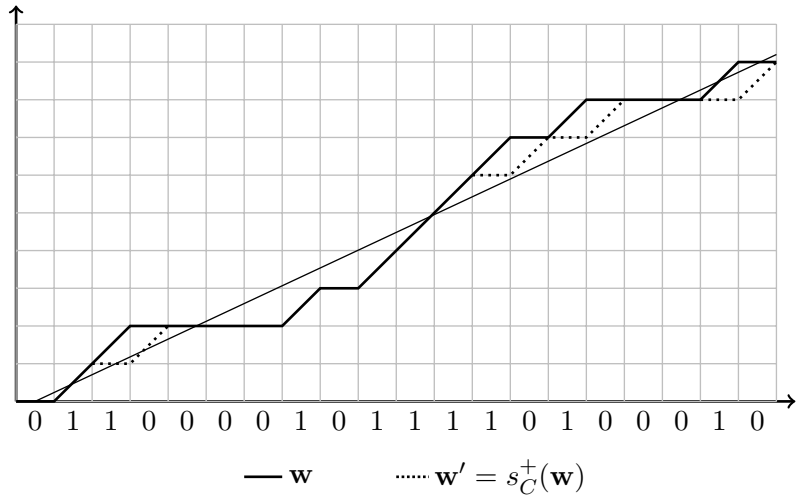

The following claim is immediate:
\begin{claim}\label{claim:SqueezeFrequency}
The operation of squeezing does not change the letter frequencies.
\end{claim}

Note that, by \autoref{lem:continuity} and \autoref{lem:corridorLemma}, we have
\begin{claim}
Let $\infw{u}$ be a binary word with $\freq[\infw{u}]{1} = \alpha$. Then for each $m$ there exist $i$ and $j$ such that $|u_i\cdots u_{i+m-1}|_1 = \lfloor \alpha m\rfloor$ and
$|u_j\cdots u_{j+m-1}|_1=\lceil \alpha m\rceil$.
\end{claim}

Let us now prove the main technical lemma relating $C$-squeezings and abelian closures.
\begin{lemma}\label{lem:uppersqueezing} For a binary word $\infw{x}$ we have $s^{+}_{C}(\infw{x}) \in \abclsr{\infw{x}}$.\end{lemma}

\begin{proof}
Assume the converse. Let $\infw{x}' = s^{+}_{C}(\infw{x})$. Then, due to
\autoref{lem:corridorLemma} there exists a factor $a'_i\cdots a'_{j-1}$ such that for each $k$
we have $|a'_i\cdots a'_{j-1}|_1 > |a_k\cdots a_{k+j-i-1}|_1$ 
(the case of $<$ is symmetric).

First we remark that the switches inside the factor (at positions $i+1,\dots,j-1$) do not change the Parikh vector of the factor. So, to change the Parikh vector, we must have a switch at position $i$ or $j$ (or both). 
Secondly, note that we have a switch at position $\ell$ if and only if
$g_{\infw{x}}(\ell) \neq g_{\infw{x}'}(\ell)$.
\begin{enumerate}[leftmargin=*]
\item 
Switch at $i$ and not in $j$. 


In this case we must have $g_{\infw{x}}(i)>\alpha i + C$, $a'_i=1$, $a_i=0$.
Notice that we have $|a'_i\cdots a'_{j-1}|_1 > \lceil \alpha(j-i)\rceil$ (recall
that $\infw{x}$ always contains a binary word with weight $\lceil n\alpha \rceil$. This in turn
means that $g_{\infw{x}'}(j) > \alpha j +C$ (due to frequency). By the conditions of Case 1 we have that $g_{{\infw{x}}'}(j)=g_{\infw{x}}(j)$, which means that $a_{j-1}a_j\neq 10$. If $a_{j-1}=0$, then by taking $k=i-1$ we get an abelian equivalent factor in $\infw{x}$
(see \autoref{fig:C-balA}). If $a_{j-1}=1$, then $a_j=1$ and we can take $k=i+1$  (see \autoref{fig:C-balB}). 

\begin{figure}
\centering
\begin{subfigure}{0.45\textwidth}
\centering
\begin{tikzpicture}[scale=0.5]

\tikzstyle{every node}=[shape=rectangle,fill=none,draw=none,minimum size=0cm,inner sep=2pt]

    \tikzstyle{every path}=[draw=black,line width = 0.5pt]    
\draw[gray!50] (0,0) grid (11,6);

    \tikzstyle{every path}=[draw=black,line width = 1pt]      
\draw (1,1)  --  (2,2) -- (3,2);
\draw (9,4)  --  (8,4);

\draw[gray!90,-latex,shorten >=1pt] (2,1) to (9,4);
\draw[gray!90,-latex,shorten >=1pt] (1,1) to (8,4);

\draw (3,-1)  --  (4,-1);
\draw[densely dotted] (6,-1)  --  (7,-1);

 \node[fill=white] at (4.5,-1){$u$};
 \node[fill=white] at (7.5,-1){$u'$};

\draw[densely dotted]  (1,1)  --  (2,1) -- (3,2);

 \node[fill=none] at (1.25,1.7){\footnotesize{$1$}};
  \node[fill=none] at (2.5,2.3){\footnotesize{$0$}};
   \node[fill=none] at (1.5,0.7){\footnotesize{$0$}};
    \node[fill=none] at (2.7,1.3){\footnotesize{$1$}};
     \node[fill=none] at (8.5,4.3){\footnotesize{$0$}};

 \filldraw (2,1) circle(2pt);
 \filldraw (9,4) circle(2pt);

 \node[fill=white] at (2,-0.5){$i$};
 \node[fill=white] at (9,-0.5){$j$};

\end{tikzpicture}
\caption{If $a_{j-1}=a'_{j-1}=0$, take $k=i-1$.}
\label{fig:C-balA}
\end{subfigure}\quad\quad
\begin{subfigure}{0.45\textwidth}
\centering
\begin{tikzpicture}[scale=0.5]

\tikzstyle{every node}=[shape=rectangle,fill=none,draw=none,minimum size=0cm,inner sep=2pt]

    \tikzstyle{every path}=[draw=black,line width = 0.5pt]    
\draw[gray!50] (0,0) grid (11,6);

    \tikzstyle{every path}=[draw=black,line width = 1pt]      
\draw (1,1)  --  (2,2) -- (3,2);
\draw (8,3)  --  (10,5);

\draw[gray!90,-latex,shorten >=1pt] (2,1) to (9,4);
\draw[gray!90,-latex,shorten >=1pt] (3,2) to (10,5);

\draw (3,-1)  --  (4,-1);
\draw[densely dotted] (6,-1)  --  (7,-1);

 \node[fill=white] at (4.5,-1){$u$};
 \node[fill=white] at (7.5,-1){$u'$};

\draw[densely dotted]  (1,1)  --  (2,1) -- (3,2);

 \node[fill=none] at (1.25,1.7){\footnotesize{$1$}};
  \node[fill=none] at (2.5,2.3){\footnotesize{$0$}};
   \node[fill=none] at (1.5,0.7){\footnotesize{$0$}};
    \node[fill=none] at (2.7,1.3){\footnotesize{$1$}};
     \node[fill=none] at (8.7,3.3){\footnotesize{$1$}};
      \node[fill=none] at (9.7,4.3){\footnotesize{$1$}};
   
 \filldraw (2,1) circle(2pt);
 \filldraw (9,4) circle(2pt);

 \node[fill=white] at (2,-0.5){$i$};
 \node[fill=white] at (9,-0.5){$j$};

\end{tikzpicture}
\caption{If $a_{j-1}=a'_{j-1}=1$ and $a_{j}=a'_{j}=1$, take $k=i+1$.}
\label{fig:C-balB}
\end{subfigure}\quad\quad
\caption{Case 1 of the proof of \autoref{lem:uppersqueezing}: $a_{i-1}a_i=10$, $a'_{i-1}a'_i=01$.}\label{fig:squeezing_proof}
\end{figure}
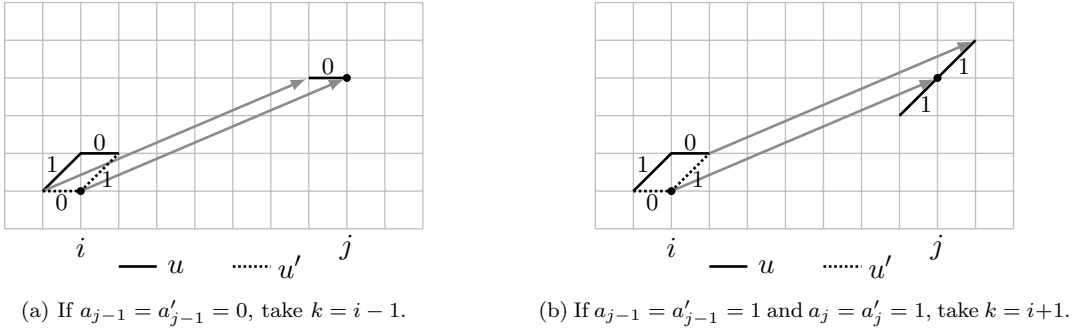

\item Switches at both $i$ and $j$. 

In this case $g_{\infw{x}}(i)>\alpha i + C$ and $g_{\infw{x}}(j)>\alpha j + C$, then the Parikh vector does not change (we can take $k=i$).

\item Switch at $j$ and not in $i$. 
In this case $a_{j-1}=1$ and $a'_{j-1}=0$, so  $|a'_i\cdots a'_{j-1}|_1<|a_k\cdots a_{k+j-i-1}|_1$, which contradicts our assumption of $a'_i\cdots a'_{j-1}$ being heavier than the factors of $\infw{x}$.
\qedhere
\end{enumerate}
\end{proof}

To prove \autoref{Prop:C-bal}, we prime the situation as follows. Notice that the property of
a word being $C$-balanced for some $C$ is equivalent to the property that its graph lies between two lines%
\footnote{If, e.g., the graph goes above the line $y = \alpha x + C$, say $g_{\infw{x}}(i) > \alpha i + C$, then the prefix of length $i$ has weight at least
$\lceil i \alpha \rceil + C$. But $\infw{x}$ also contains a factor with weight
$\lfloor i \alpha \rfloor$ contradicting the $C$-balancedness of $\infw{x}$.}
$y=\alpha x + C_1$ and $y=\alpha x + C_2$ for some $C_1, C_2\in\mathbb{R}, C_1 < C_2$.
Here we choose $C_1$ and $C_2$ to be the largest and the smallest possible, i.e. $C_1=\limsup \{C \colon g_{\infw{x}}(x) \geq \alpha x + C\}$ and
$C_2=\liminf \{C \colon g_{\infw{x}}(x)\leq \alpha x + C\}$.  
Notice that the line $\alpha x + C'$ contains at most one integral point since $\alpha$ is
irrational.

\begin{definition}
Let $\infw{x}$ be an infinite binary word that is $C$-balanced. The \emph{width} of a factor
$v$ of $\infw{x}$ is defined as
$\sup_{x}(g_{\infw{x}}(x)-\alpha x) - \inf_{x}(g_{\infw{x}}(x)-\alpha x)$, where the maximum and the minimum is chosen from the positions $x$ in an occurrence of $v$ in ${\infw{x}}$.
\end{definition}
Similarly we can define a width for an infinite word. In fact, a width of a factor depends only on the frequency of letters in $\infw{w}$ and not on the word $\infw{w}$ itself. Clearly, a width of a factor $v$ of an infinite word $\infw{w}$ cannot be bigger than the width of the $\infw{w}$.

\begin{proof}[Proof of \autoref{Prop:C-bal}]
The proof is based on the operation of upper $C_2-\varepsilon$-squeezing; in fact, choosing different values of $\varepsilon$, we can get different minimal subshifts from the abelian closure of the initial word. 
 
 The proof is split into several claims. 

\begin{claim}\label{claim4}
For each small enough $\varepsilon > 0$ there exist infinitely many points of the graph $g_{\infw{x}}$ in
the stripe between $y=\alpha x + C_2-\varepsilon$ and $y=\alpha x + C_2$. Moreover, the gaps between these points are bounded.
\end{claim}
\begin{proof}
(See \autoref{fig:C-bal}.)
Since $C_2 = \liminf \{C \colon g_{\infw{x}}(x)\leq \alpha x + C\}$, there exists by definition a point
$x$ such that $g_{\infw{x}}(x) > \alpha x + C_2-\varepsilon/2$. Symmetrically, there exists a
point $x'$ such that $g_{\infw{x}}(x') < \alpha x + C_1+\varepsilon/2$. Without loss of
generality we may assume that $x < x'$. Now since the word is uniformly recurrent, a factor
abelian equivalent to $a_{x}\dots a_{x'-1}$ occurs with bounded gaps; let $N(\varepsilon)$
denote an upper bound on the gaps. For each large enough occurrence of this factor its initial point is
above the line $y= \alpha x + C_2-\varepsilon$, and its final point is below the line
$y= \alpha x + C_1+\varepsilon$. Indeed, if the initial point was below
the line $y= \alpha x + C_2-\varepsilon$ for arbitrarily large $x$, then the final point would lie below $\alpha x + C_1 - \varepsilon/2$ for arbitrarily large $x$.
\end{proof}

We will now use an operation of upper $(C_2-\varepsilon)$-squeezing of $\infw{x}$. For simplicity, assume that $\varepsilon < \alpha$ (this assumption is made in order to flip all the points on the stripe of width $\varepsilon$), and that $C_2-\varepsilon-C_1>1$ (this corresponds to Sturmian width; we make this assumption to guarantee that the flipped points remain above the line $\alpha x + C_1$). Due to \autoref{lem:uppersqueezing}, if
${\infw{x}}'=s^{+}_{C_2-\varepsilon}({\infw{x}})$, then ${\infw{x}}'\in \abclsr{{\infw{x}}}$.



\begin{claim}\label{claim5}
Any uniformly recurrent word from $\soc{{\infw{x}}'}$ is different from words in
$\soc{{\infw{x}}}$.
\end{claim}
\begin{proof}
Indeed, there is no factor from the proof of \autoref{claim4} in ${\infw{x}}'$ (their width is greater than $C_2-\varepsilon-C_1$, so they do not fit the stripe of ${\infw{x}}'$).
\end{proof}

We now claim that varying $\varepsilon$ in an appropriate way we obtain infinitely many minimal
subshifts in $\abclsr{\infw{x}}$. 

Given $\varepsilon$, let $N(\varepsilon)$ be the constant from \autoref{claim4}
(it is actually given by uniform recurrence), giving an upper bound for the points of the graph of $\infw{x}$ above the line $y=\alpha x +C_2-\varepsilon$.

Now choose $\varepsilon_1 < \varepsilon$ to be the constant such that the points of the grid in
the stripe between the lines  $y=\alpha x +C_2-\varepsilon_1$ and $y=\alpha x +C_2$ are at
distance at least $2N(\varepsilon)$ (the value of $\varepsilon_1$ is given by the irrational
value $\alpha$).

By the choice of $N(\varepsilon)$ and $\varepsilon_1$, we have points of the graph
$g_{\infw{x}}$ in the stripe between $y=\alpha x +C_2-\varepsilon$ and $y=\alpha x +C_2-\varepsilon_1$ with gap at most $2N(\varepsilon)$.

We will now prove that for the word
${\infw{x}}''=s^{+}_{C_2-\varepsilon_1}({\infw{x}})$ we have that each 
uniformly recurrent point in $\soc{{\infw{x}}''}$ is different from any 
point from $\soc{{\infw{x}}'}$. Indeed, from what we just proved above, 
$g_{\infw{x}}$ (and hence $g_{\infw{x}''}$) has points between $y=\alpha x +C_2-\varepsilon$ and $y=\alpha x +C_2-\varepsilon_1$ with gap at most 
$2N(\varepsilon)$. Let $\tilde{\varepsilon}$ be such that the integer 
points in the stripe between $y=\alpha x +C_2-\varepsilon$ and $y=\alpha x +C_2-\varepsilon+\tilde{\varepsilon}$ are with gap at least 
$4N(\varepsilon)$ (the value
$\tilde{\varepsilon}$, like $\varepsilon_1$, is defined by the irrational $\alpha$). So, the
graph of ${\infw{x}}$ (and hence of ${\infw{x}}''$) has points between
$y=\alpha x +C_2-\varepsilon+\tilde{\varepsilon}$ and $y=\alpha x +C_2-\varepsilon_1$ with gap
at most $4N(\varepsilon)$. (See \autoref{fig:Bal1} for an illustration of the situation.) Now, by an argument symmetric to \autoref{claim4} applied for
$\tilde{\varepsilon}$, we get that there is an upper bound $N(\tilde{\varepsilon})$ for the gaps between the points of $g_{\infw{x}}$ between the lines $y=\alpha x +C_1$ and $y=\alpha x +C_1+\tilde{\varepsilon}$. Taking $\tilde{N}=\max(4N(\varepsilon), N(\tilde{\varepsilon}))$ and considering points in the two stripes (between the lines  $y=\alpha x +C_1$ and
$y=\alpha x +C_1+\tilde{\varepsilon}$ and the lines
$y=\alpha x +C_2-\varepsilon+\tilde{\varepsilon}$ and $y=\alpha x +C_2-\varepsilon_1$), we
have factors of length at most $\tilde{N}$ of width at least $C_2-C_1-\varepsilon$ with gap
at most $N'$. Since there are only finitely many such factors, one of them occurs with
bounded gap in ${\infw{x}}''$, and hence in any word from $\soc{{\infw{x}}''}$. On the other
hand, these factors are too wide to fit into the stripe for ${\infw{x}}'$ (which is of width
$C_2-C_1-\varepsilon$), so in fact $\soc{{\infw{x}}'}$ and $\soc{{\infw{x}}''}$ do not
intersect, and hence $\soc{{\infw{x}}''}$ contains a new minimal subshift which is in
$\abclsr{\infw{x}}$.

We continue this line of reasoning taking $\varepsilon_1$ instead of $\varepsilon$ etc., each time getting a new minimal subshift in the abelian closure of $w$. This concludes the proof.
\end{proof}

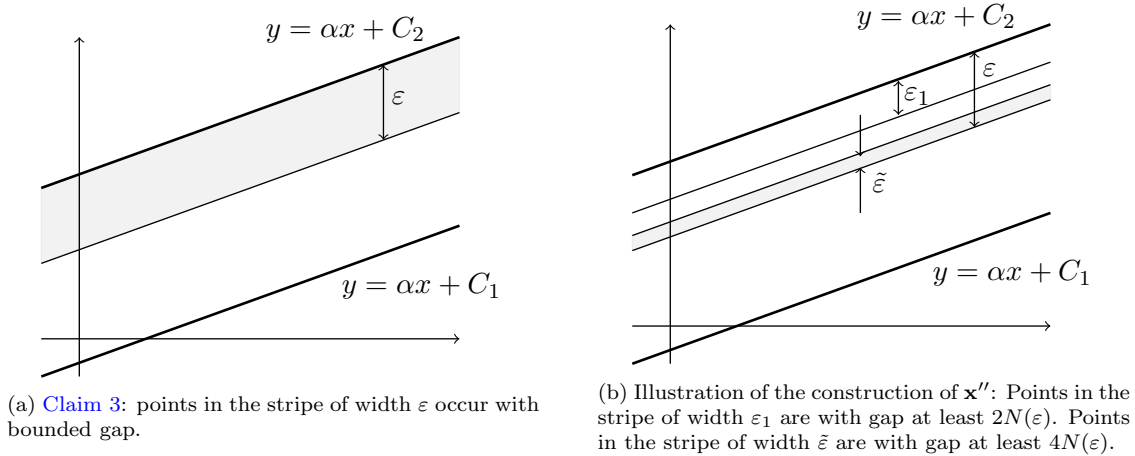
\begin{figure}
\centering
\begin{subfigure}{0.45\textwidth}
\centering
\begin{tikzpicture}[scale=1]
\tikzstyle{every node}=[shape=rectangle,fill=none,draw=none,minimum size=0cm,inner sep=2pt]
    \tikzstyle{every path}=[draw=black,line width = 0.5pt]

    \draw[->] (-0.5,0)  to   (5,0);
    \draw[->] (0,-0.5)  to   (0,4);

    \tikzstyle{every path}=[draw=black,line width = 1pt]    
    \draw (-0.5,-0.5) to (5,1.5);
    \draw[name path = A] (-0.5,2) to (5,4);

    \tikzstyle{every path}=[draw=black,line width = 0.5pt]
    \draw[name path = B] (-0.5,1) to (5,3);
\tikzfillbetween[of=A and B,on layer=bg]{gray, opacity=0.1};

 \node at (4.5,0.7){$y=\alpha x + C_1$};
 \node at (3.5,4.1){$y=\alpha x + C_2$};

\draw[<->] (4,3.63) to (4,2.64);

\node at (4.2,3.2){$\varepsilon$};

\end{tikzpicture}
\caption{\autoref{claim4}: points in the stripe of width $\varepsilon$ occur with bounded gap.}
\label{fig:C-bal}
\end{subfigure}\quad\quad
\begin{subfigure}{0.45\textwidth}
\centering
\begin{tikzpicture}[scale=1]
\tikzstyle{every node}=[shape=rectangle,fill=none,draw=none,minimum size=0cm,inner sep=2pt]
    \tikzstyle{every path}=[draw=black,line width = 0.5pt]

    \draw[->] (-0.5,0)  to   (5,0);
    \draw[->] (0,-0.5)  to   (0,4);

    \tikzstyle{every path}=[draw=black,line width = 1pt]    
    \draw (-0.5,-0.5) to (5,1.5);
    \draw (-0.5,2) to (5,4);

    \tikzstyle{every path}=[draw=black,line width = 0.5pt]
    \draw[name path = A] (-0.5,1) to (5,3);
    \draw (-0.5,1.5) to (5,3.5);
    \draw[name path = B] (-0.5,1.2) to (5,3.2);
    
    \tikzfillbetween[of=B and A, on layer=bg]{gray, opacity=0.1};

 \node at (4.5,0.7){$y=\alpha x + C_1$};
 \node at (3.5,4.1){$y=\alpha x + C_2$};

\draw[<->] (4,3.63) to (4,2.64);

\node at (4.2,3.4){$\varepsilon$};

\draw[->] (2.5,1.5) to (2.5,2.1);
\draw[->] (2.5,2.8) to (2.5,2.25);

\node at (2.75,1.9){$\tilde{\varepsilon}$};

\draw[<->] (3,2.8) to (3,3.25);

\node at (3.25,3.05){$\varepsilon_1$};

\end{tikzpicture}
\caption{Illustration of the construction of $\infw{x}''$: Points in the stripe of width $\varepsilon_1$ are with gap at least $2N(\varepsilon)$. Points in the stripe of width $\tilde{\varepsilon}$ are with gap at least $4N(\varepsilon)$.}
\label{fig:Bal1}
\end{subfigure}
\caption{An illustration for the proof of \autoref{Prop:C-bal}.}
\end{figure}

\section{Some structural results on binary abelian closures} \label{sec:tools}

The results
presented in this section will be used as tools in proving the main result of the subsequent
section (and the last and the hardest case of the main theorem), though they might have independent interest.  We consider certain operations on binary words, and consider how they
affect abelian closures. We first discuss morphic images of abelian closures. We then
define an operation, which resembles an elementary \emph{cellular automaton} on
right-infinite words (see the precise definition in \autoref{subsec:CA}). And finally, we give a certain description of non-Sturmian binary words in terms of Standard pairs. 

\subsection{Morphisms and binary abelian closures}

Morphisms are an essential tool in the study of combinatorics on words. In this subsection we study the interaction between abelian closures and morphisms. As the main
result of this subsection, we show that, given a \emph{Sturmian morphism $f$} (see definition below), if $\infw{z} \in \abclsr{\infw{y}}$ then
$f(\infw{z}) \in \abclsr{f(\infw{y})}$.
However, in general this property does not hold even for binary alphabets, as is illustrated by the following example. 

\begin{example}\label{ex:morphicImage}
Take $\infw y = (0011)^{\omega}$ and observe that $\infw z = (01)^{\omega} \in \abclsr{\infw{y}}$ ($\infw{z}$ is a periodic Sturmian of slope $1/2$). Define $f$ by $f(0) = 100001$ and $f(1) = 010$. Hence
\begin{equation*}
f(\infw{y}) = (100001100001010010)^{\omega} \text{ and } f(\infw{z}) = (100001010)^{\omega}.
\end{equation*}
Observe now that the length $5$ factors of $f(\infw{y})$ all have at most two occurrences of $1$.
On the other hand, $f(\infw{z})$ contains the factor $10101$ which has three occurrences of
$1$. Hence $f(\infw{z}) \notin \abclsr{f(\infw{y})}$.
\end{example}

Let us now recall Sturmian morphisms and standard morphisms. For a concise
treatment of these morphisms, see \cite[\S2.3]{MR1905123}. We then consider the abelian
closures of morphic images of binary words.

A morphism $\varphi:\{0,1\}^*\to \{0,1\}^*$ is called \emph{Sturmian} if, for
each Sturmian word $\infw x$, the word $\varphi(\infw x)$ is Sturmian.
We shall employ a striking result of Mignosi and S\'e\'ebold
\cite{MignosiSeebold1993:morphismes_sturmiens_et_regles_de_rauzy} which
characterizes the set of Sturmian morphisms as the finitely generated monoid
with generators
\begin{equation*}
    D:\begin{cases}0\mapsto 01\\ 1\mapsto 0\end{cases}; \quad
    E: \begin{cases}0\mapsto 1\\ 1\mapsto 0\end{cases}; \quad
    G: \begin{cases}0\mapsto 10\\ 1\mapsto 0\end{cases}.
\end{equation*}

A morphism $\varphi$ is called \emph{standard} if, for any (unordered) standard pair $\{u,v\}$,
the pair $\{\varphi(u),\varphi(v)\}$ is standard. Standard moprhisms were studied by
A. de Luca in \cite{deLuca1996:on_standard_sturmian_moprhisms}. In that article, the set of
standard morphisms is characterized as the finitely generated monoid with generators $\{D,E\}$
defined above.

We also recall the following result from
\cite[Thms.~8,~9,~12]{deLuca1996:on_standard_sturmian_moprhisms}. 
\begin{theorem}\label{thm:AdeLucaStandarMorphism}
A morphism $\varphi$ is standard if and only if $\varphi(0) = x$ and $\varphi(1) = y$ for some
$x,y$ with $\{x,y\}$ an unordered standard pair.
\end{theorem}

We are now ready to show the main result of this subsection.
\begin{proposition}\label{prop:AbSsClosedUnderSturmian}
Let $\varphi$ be a Sturmian morphism. Then $\infw{z} \in \abclsr{\infw{y}}$ implies
$\varphi(\infw{z}) \in \abclsr{\varphi(\infw{y})}$.
\end{proposition}

\begin{proof}
It suffices to show that the claim holds for the morphisms $E$, $D$ and $G$. To
conclude from that, we may proceed by a simple induction on the length of the shortest
representation of a Sturmian moprhism as a composition of these generating morphisms.

The case of $E$ is trivial. We prove the claim for $D$, the
case of $G$ being symmetric. Let $z$ be a factor of $D(\infw{z})$. We show that there is
an abelian equivalent factor in $D(\infw{y})$ such that $z' \sim z$.
By the form of $D$, $z$ can be written as one of the following forms: $D(w)$, $D(w)0$,
$1D(w)$, and $1D(w)0$, where $w$ is some factor of $\infw{z}$. The claim is easily seen to
hold in the case $w=\eps$, so we assume that $w\neq \eps$.

Assume first that
$z = D(w)$, with $w\in \mathcal{L}(\infw{z})$. Then $\infw{y}$ has a factor $w' \sim w$,
so $D(w') = z' \sim  z$. Assume second that $z = D(w)0$.
There exists a factor $w'x$, $x\in \{0,1\}$, of $\infw{y}$ with $w' \sim w$. By the form of $D$, $D(w'x)$
begins with $D(w')0 = z'$, with $z'\sim z$.

Assume third that $z = 1 D(w)$. Then $0w$ is a factor of $\infw{z}$. Let us first show
that there exists a factor $w' \in \infw{y}$ with $w' \sim w$, and $0w'$ or $w'0$ occurs
in $\infw{y}$. Indeed, since $0w$ occurs in $\infw{z}$, there is a corresponding abelian
equivalent factor $t$ in $\infw{y}$. Consider an occurrence of $w'1$. Assuming that $t$
occurs before $w'1$, by a sliding window argument, there is a factor of the form $0t'1$,
where $0t' \sim t$ and $t'1 \sim w'1$. Thus $0t'$ is the factor we are looking for. The
case that $t$ occurs after $w'1$ is symmetric. Now if $0w'$ occurs in $\infw{y}$, we have
$1D(w') = z'$ and we are done. If $w'0$ occurs in $\infw{y}$, then, by the form of
$D$, we may write $D(w') = 0 x D(w'')$ for some $x\in \{\eps,1\}$ and $w''\in \words$.
Now $xD(w'')0 \sim D(w')$, and thus $xD(w'')01 = z' \sim 1D(w) = z$.

Assume finally that $z = 1 D(w) 0$. Notice again that $0w$ occurs in $\infw{z}$.
Proceeding as in the previous case, there exists a factor $w'\sim w$ such that $0w'$ or
$w'0$ occurs in $\infw{y}$. This time we may choose $z' = D(w'0)$ or $z' = D(0w')$. 
\end{proof}

\subsection{A selfmap on binary abelian closures}\label{subsec:CA}
In this section define a selfmap on binary abelian closures. This mapping has interesting
dynamics, as we shall shortly see. We present the observations as having independent
interest, but the main result of this section will be crucial in our subsequent
constructions.

\begin{definition}\label{def:F-operationCA}
Let us define the operation $T:\{0,1\}^{\N}\to \{0,1\}^{\N}$ by the following rule:
$T(\infw x)$ is obtained from $\infw x$ by replacing each occurrence of $10$
with $01$. Let us further define $F = \sigma \circ T$. Thus $F$ operates on a infinite
binary word by first flipping each occurrence of $10$ to $01$ ($T$), and second removing the first letter ($\sigma$).
\end{definition}

Observe that the operation $T$ is a simplified version of the $C$-squeezing operation used in \autoref{sec:C-bal}.

\begin{remark}
The operation $T$ can be defined for bi-infinite words (words indexed by the set of integers)
as well. In this setting $T$ is a
\emph{cellular automaton} known as the \emph{Traffic cellular automaton}. It is
\emph{Rule $184$} in the system of S. Wolfram \cite{Wolfram1983:statistical_mechanics_of_cellular_automata}.

\end{remark}

Let us show that the mapping $F$ is indeed a selfmap on a binary subshift.
\begin{lemma}\label{lem:TrafficPreservesAbelianSubshift}
For any binary word $\infw x\in \{0,1\}^{\N}$, we have
$F(\infw x) \in \abclsr{\infw x}$.
\end{lemma}
\begin{proof}
Assume the contrary, that $F(\infw x)$ has a factor which
is either heavier or lighter than all factors of $\infw x$. Write
$\infw x = a_0a_1a_2\cdots$ and
$T(\infw x) = b_0b_1b_2\cdots$, where $a_i,b_i\in\{0,1\}$.
Assume first that the factor $u = b_i\cdots b_j$ of $T(\infw{x})$, with $i\geq 1$ (recall
that $F(\infw{x})$ is obtained by removing the first letter of $T(\infw{x})$), is heavier
than all factors of $\infw x$ of the same length. We consider how $u$ is generated from
$\infw{x}$ under $T$:

\begin{center}
\newlength{\Csep}
\setlength{\Csep}{16pt}
\begin{tabular}{r *{8}{C{\Csep}|} C{\Csep}}
\cline{2-10}
   $\infw{x}$:  & \cdots & a_{i-1} & a_{i}  & a_{i+1} & \cdots & a_{j-1} & a_j & a_{j+1} & \cdots \\ \cline{2-10} \\\cline{2-10}
   $T(\infw{x})$: & \cdots & * & b_{i} & b_{i+1} & \cdots & b_{j-1} & b_j & * & \cdots \\\cline{2-10}
\end{tabular}
\end{center}
Consider the factor $v = a_i\cdots a_j$ of $\infw x$. Since $u$ is heavier than $v$, we
must have $a_{i-1}a_i = 10$ and $a_ja_{j+1}\neq 10$ by the definition of $T$.
Moreover, we see that $|u|_1 \leq 1+|v|_1$.
If $a_j=0$, then $|a_{i-1} \cdots a_{j-1}|_1 = 1+|v|_1$, so $u$ is not heavier than this
factor of $\infw{x}$. Now if $a_j = 1$, then necessarily $a_{j+1}=1$ and thus $|a_{i+1} \cdots a_{j+1}|_1 = 1+|v|_1$, and again, $u$ is not heavier than this factor of
$\infw{x}$. In either case, $u$ has an equal weight and length corresponding factor in
$\infw x$, a contradiction.

The case of $b_i\cdots b_j$ being lighter than all other factors is symmetric.
One simply notes that in this case necessarily $a_ja_{j+1} = 10$ and $a_{i-1}a_i \neq 10$. 
\end{proof}

The proof above is similar to the proof of \autoref{lem:uppersqueezing}, but the shift operation
in the definition of $F$ cannot be removed.
\begin{remark}
We remark that $T(\infw x)$ is not necessarily an element of $\abclsr{\infw x}$.
For example, consider the Sturmian word $0\infw f = 0010\dots$, where $\infw f$ is the
Fibonacci word. Observe that $T(0\infw f)$ begins with
$0001$, hence $T(0\infw f) \notin \soc{0\infw{f}} = \abclsr{0\infw f}$.
\end{remark}

Clearly $F(\infw x)$ is uniformly recurrent if $\infw x$ is. Further, if $\infw x$ has
uniform letter frequencies, then $F(\infw x)$ has the same letter frequencies as $\infw x$
by \autoref{lem:uniformFrequencies}. Moreover, if $\infw x$ is non-balanced,
then so is $F(\infw x)$. This can be shown by using similar arguments
as in the above proof.

We now consider the operation of iterating $F$ on a word. For this, we need some
terminology.
\begin{definition}
A word $\infw y$ is a \emph{preimage of order $n$} of a word $\infw x$, if
$F^n(\infw y) = \infw x$. We say that $\infw{x}$ \emph{has a preimage of order $n$}, if
such a $\infw{y}$ exists.
\end{definition}

\begin{lemma}\label{lem:preimageOfOrderN}
Let $\infw x$ be a binary word containing the factor $11(01)^n00$ for some
$n\geq 0$. Then $\infw x$ has no preimage of order $n+1$.
\end{lemma}
\begin{proof}
If $\infw x$ does not have a preimage under $F$, then there is nothing to prove.
Assume it has a preimage $\infw y$ under $F$. We thus have $T(\infw y) = \infw{x}'$,
where $\infw{x} = \sigma(\infw{x}')$. Consider the position in $\infw y$ corresponding to
where $11(01)^n00$ occurs in $\infw{x}'$ (note that $\infw{x}'$ also contains $11(01)^n00$):
\begin{center}
\begin{tabular}{r C{15pt}|C{6pt}|C{6pt}|C{6pt}|C{15pt}|C{6pt}|C{6pt}|C{6pt}|C{15pt}}\cline{2-10}
   $\infw{y}$: &  \cdots & * & a & b  & \cdots &  c & d & * & \cdots \\\cline{2-10}
   \\\cline{2-10}
 $\infw{x}'$: & \cdots & 1 & 1 & 0 &  \cdots &  1 & 0 & 0 & \cdots\\\cline{2-10}
\end{tabular}
\end{center}
The only option is that $ab=11$ and $cd = 00$: If $a=0$,
then either $a$ stays in the same position or is moved one step to the left
depending on which letter precedes $a$ in $\infw y$.
This is impossible, since $\infw{x}'$ has $1$ in both positions. Furthermore 
if $b=0$, then the letter $a=1$ would be shifted by $T$ to the right by one
position, which is also not possible, as $\infw{x}'$ has $0$ in that position.
Similar arguments show that $cd=00$.

Observe that the arguments showing $ab=11$ and $cd=00$ are independent of each other.
Now if $n=0$, the above observation poses a contradiction: we should have $1=a=c=0$.
Thus $\infw x$ has no preimage under $F$.
For $n\geq 1$, we deduce from the
above that $\infw y$ contains the factor $11(01)^{m}00$ for some $m \leq n-1$. By
induction, $\infw y$ has no preimage of order $m+1$, so that $\infw x$ has no preimage of order $m+2 \leq n+1$, as was to be shown.
\end{proof}

The dynamics of the mapping $F$ will be of interest to us in our later considerations.
The following proposition is the main result of this section.

\begin{proposition}\label{prop:finitelyManyIterationsAll1sIsolated}
Let $\infw x$ be a binary word with $\freq{1} < 1/2$. Then there exists
an integer $n\geq 0$ such that all $1$s are isolated in $F^n(\infw x)$,
that is, $11 \notin \lang{F^n(\infw x)}$.
\end{proposition}
\begin{proof}
If all $1$s are isolated in $\infw x$ we may choose $n=0$. Assume that
$11$ occurs in $\infw x$. Due to our assumption $\freq{1} < 1/2$, there must exist a
factor $v$ of maximal length for which $\freq[v]{1} > 1/2$ and, further, in which $11$
occurs. We call such a factor of $\infw{x}$ \emph{exceptional}. Note that an exceptional factor
has length at least $3$, since $110$ must occur in $\infw{x}$ under the assumptions. Now
any occurrence of an exceptional factor $v$ (occurring after the prefix of
length $2$) must be preceded and followed by $00$ in $\infw{x}$.
Otherwise $\infw{x}$ contains a factor of length $|v|+2$ with frequency at least
$\frac{|v|_1+1}{|v|+2}>1/2$ and which contains $11$.

We partition the rest of the proof into a couple of claims.
\begin{claim}
For any exceptional factor $u$ of $F(\infw{x})$, there exists an exceptional factor $v$ of $\infw{x}$ such that $|u| \leq |v|$.
\end{claim}
\begin{proof}
Write $\infw{x} = a_0a_1\cdots$ and $T(\infw{x}) = b_0b_1\cdots$.
Let $u = u_1 11  u_2$ be an exceptional factor of $F(\infw{x})$, where
$u_1,u_2\in \{0,1\}^*$. Actually, $u_2\neq \eps$:
notice again that $u$ is followed by $00$ in $F(\infw{x})$ so $u$ cannot end with $11$, as
otherwise $F(\infw{x})$ would contain the factor $1100$. \autoref{lem:preimageOfOrderN} would then imply that $F(\infw{x})$ does not have a preimage, which is absurd.

Let us depict an occurrence of
$u = b_i\cdots b_j$ in $T(\infw{x})$, with $i\geq 1$.

\begin{center}
\begin{tabular}{r C{15pt} *{2}{|C{16pt}} | C{15pt} *{3}{|C{6pt}} |C{15pt} *{3}{|C{16pt}} |C{15pt}}
\cline{2-13}
   $\infw{x}$: &  \cdots & a_{i-1} &  a_i  & \cdots & * & a  &  b & \cdots & a_j & a_{j+1} & * &\cdots
   \\\cline{2-13}
   \\\cline{2-13}
 $T(\infw{x})$: & \cdots & * & b_i & \cdots & 1 &  1 & * & \cdots & b_j & 0 & 0 & \cdots
 \\\cline{2-13}
\end{tabular}
\end{center}
Here we allow $F(\infw{x})$ to begin with $u$, so that $T(\infw{x})$ would begin with
$0u$ or $1u$ (this does not matter in the following argument). Observe that this
particular occurrence of $u$ depends only on the factor $a_{i-1}\cdots a_{j+1}$. We therefore have $|u|_1 \leq |a_{i-1}\cdots a_{j+1}|_1$ by the form of $T$.

We now have that $ab=11$ and $a_ja_{j+1}=00$ by the same arguments as used in
\autoref{lem:preimageOfOrderN}. It then follows that
$|u|_1 \leq |a_{i-1}\cdots a_{j-1}|_1$ as we had $a_ja_{j+1}=00$. Now
$a_{i-1}\cdots a_{j-1}$ contains an occurrence of $11$ and the frequency of $1$ is larger
than $1/2$. Hence $\infw{x}$ contains an exceptional factor of length at least $|u|$, as was claimed.
\end{proof}
As a consequence of the above claim, for any $m\geq 0$ and for any exceptional factor $u$ of $F^{m}(\infw{x})$, there exists an exceptional factor $v$ of $\infw{x}$ such that $|u| \leq |v|$.

\begin{claim}
There exists an integer $m$ such that any exceptional factor of $F^{m}(\infw{x})$ is shorter
than any exceptional factor in $\infw{x}$.
\end{claim}
\begin{proof}
Let $v$ be an exceptional factor of $\infw{x}$ and take $m = \left\lfloor|v|/2\right \rfloor$.
Assume, for a contradiction, that $\infw y = F^{m}(\infw x)$ contains an exceptional factor $u$
of length $|v|$. Similar to $v$, all occurrences of $u$ in $\infw y$ are followed by $00$.
We infer that $u00$ contains a factor of the form $11(01)^k00$, where
$2(k+1)\leq |u| = |v|$. By \autoref{lem:preimageOfOrderN}, $\infw y$ has a preimage of the
order at most $k \leq |v|/2-1 < \left\lfloor |v|/2 \right\rfloor=m$, which is a
contradiction.
\end{proof}

The claim above implies that there exists an integer $n\geq 1$ such that an exceptional
factor in $F^{n-1}(\infw x)$ has length at most $3$.
The only such factors are $011$ and $110$. This implies that each occurrence of $11$ is
always followed by $000$ in $F^{n-1}(\infw{x})$. We conclude that, in the word $F^n(\infw x)$, all $1$s are isolated, which was to be proved.
\end{proof}

We shall use the following immediate corollary in our later considerations.
\begin{corollary}\label{cor:AbelianSsNonBalancedContainsIsolated1}
For a non-balanced binary word $\infw{x}$ with $\freq{1} < 1/2$, there exists a non-balanced
word $\infw{x}'$ in $\abclsr{\infw{x}}$ such that $11 \notin \lang{\infw{x}'}$.
\end{corollary}

\subsection{The structure of binary words in terms of standard pairs}\label{subsec:Standard}
In this subsection we recall structural results related to standard words and central words from
\cite[\S~2.2.1]{MR1905123}. We then prove a couple of related technical lemmas about binary words that we use in the sequel for the proof of Theorem \ref{thm:binary} in the case of non-balanced words. 

The reversal $x^R$ of a finite word $x = a_0\cdots a_n$
is $x^R = a_n \cdots a_0$. If $x^R = x$, then $x$ is called a \emph{palindrome}.
It is known that a word $w$ is central if and only if $w$ is a
(possibly empty) power of a letter, or is a palindrome which can be written in the form
$p10q = q01p$ for some palindromes $p,q$. Moreover, this factorization is unique. Further, any
palindrome prefix or suffix of a central word is central.

Let us recall an important structural result on central words. We say that the word
$w = a_1a_2 \cdots a_n$, with $a_i \in \alphabet$, has \emph{period} $k$ if $a_i = a_{i+k}$
for $i=1,\ldots,n-k$. Notice that $k = n$ is allowed in this definition.

\begin{theorem}[{\cite[Thm~2.2.11]{MR1905123}}]\label{thm:periodsOfCentral}
A word is central if and only if it has two periods $k$ and $\ell$ such that
$\gcd(k,\ell) = 1$ and $|w|=k+\ell-2$. Moreover, if $w\notin 0^*\cup 1^*$ and
$w = p10q$ with $p$ and $q$ palindromes, then $\{k,\ell\} = \{|p|+2,|q|+2\}$
and the pair $\{k,\ell\}$ is unique.
\end{theorem}

We remark the following straightforward consequence of this result.
\begin{lemma}\label{lem:centralStandardPeriods}
Let $w$ be a central word with $w \notin 0^*\cup 1^*$. Write $w01 = xy$ for some standard pair $(x,y)$. Then, writing $w = p10q$ for unique central words $p$ and $q$, we have $x = p10$ and $y = q01$. Furthermore $w x^R = xw$ and $w y^R = yw$.
\end{lemma}

\begin{proof}
Observe that $x \neq 0$ as $w$ would then be a power of a letter. Similarly $y\neq 1$. Hence
$x = s10$, $y = t01$, and $w = s10t$ for some central words $s$ and $t$. Since this factorization is unique, we have $s = p$ and $t=q$. Further, by the above theorem,
$w$ has periods $|x|$ and $|y|$. For the last claim, we observe $xw = p10q01p = w x^R$, and
$yw = q01p10q = w y^R$.
\end{proof}

We need the following two technical lemmas to argue about infinite words having distinct sets
of factors if they are products of distinct standard pairs. 
 These facts might be known by some experts in Sturmian words,
but we were unable to find references for them, so we give proofs for the sake of completeness.
In what follows, a factor $w$ of an infinite word $\infw{x}$ is called \emph{right special},
if $w0,w1\in \lang{\infw{x}}$. Similarly, $w$ is \emph{left special} if $0w$,
$1w \in \lang{\infw{x}}$. Finally, $w$ is called \emph{bispecial}, if it is both right
special and left special. Also, a set $X$ of binary words is called balanced if $u,v \in X$
with $|u| = |v|$ implies that $||u|_1 - |v|_1|\leq 1$. It is known that for a balanced set
$X$ that is factor closed (i.e., $X = \cup_{x\in X} \lang{x}$), has at most $n+1$ elements
of length $n$, for each $n \in \N$ (\cite[Prop.~2.1.2]{MR1905123}). This fact will be used in
several places of the following two lemmas.

\begin{lemma}\label{lem:factorizationOfShift}
Let $\infw x$ be an infinite, recurrent, aperiodic binary word which is not Sturmian.
Then there exists a standard pair $(x,y)$ such that some shift of
$\infw x$ is a product of $x$ and $y$, and both $xx$ and $yy$ occur in the corresponding
factorization. Moreover, the shortest unbalanced pair of factors in $\infw{x}$ has length
$|xy|$.
\end{lemma}
\begin{proof}
As $\infw x$ is non-Sturmian and aperiodic, it follows that $\infw x$ contains the factors
$0w0$ and $1w1$, where $w$ is a palindrome. Furthermore, $|0w0|$ is the least length for
which such an unbalanced pair exists (this fact is implicit in the proof of
\cite[Prop.~2.1.3]{MR1905123}). Notice now that $\lang[\leq |w|+1]{\infw{x}}$ is balanced by
the minimality of $|w|$. In particular, both $0w0$ and $1w1$ are balanced. It follows that $w$
is a right special factor of some Sturmian word $\infw{s}$. Furthermore, since $w$ is a
palindrome, it is even a central word (see
\cite{deLuca1996:sturmian_words_structure_combinatorics_arithmetics} or
\cite[Prob.~2.2.7]{MR1905123}).

If $w=0^n$ for some $n\geq 0$, then $10^n1$ is the shortest block of $0$s surrounded by $1$s
occurring in $\infw x$. Since $\infw x$ is recurrent, some shift $\infw y$ of $\infw x$ begins
with $0^n1$. Note that $(0,0^n1) = (x,y)$ is a standard pair. Now $\infw y$ can be expressed
as a product of the words $x$ and $y$ in a unique way. By assumption, both $0^{n+2}$ and
$10^n1$ occur in $\infw y$. The former implies that $0^{n+2}1 = xxy$ occurs in the
factorization, and the latter implies that $0^n10^n1 = yy$ occurs in the factorization.

We are left with the case that $w\notin 0^*\cup 1^*$. Now we may write $w01 = xy$ for a
(unique) standard pair $(x,y)$.
Then $x=p10$ and $y = q01$ for some central words $p,q$ by the above lemma. We claim that a
shift $\infw y$ of $\infw x$ is a product of the words $x$, $y$, and this factorization
contains both $xx$ and $yy$. 

Next we show that $\lang[\leq |w|+1]{\infw{x}} = \lang[\leq |w|+1]{\infw{s}}$ for some
Sturmian word $\infw{s}$. Take the Sturmian word $\infw{s}$ from the beginning of this proof:
$w$ is a right special factor of $\infw{s}$. Consider the set
$X = \lang[\leq |w|+1]{\infw{x}} \cup \lang[\leq |w|+1]{\infw{s}}$. If it is balanced, then we are done, as the sets must then be equal by a counting argument. So assume that it is unbalanced. The shortest
unbalanced pair is $0w'0$, $1w'1 \in X$ (see again \cite[Prop.~2.1.3]{MR1905123}) for some
palindrome $w'$. Assume $0w'0 \in \lang{\infw{x}}$ but $1w'1 \notin \lang{\infw{x}}$ and the
converse for $\infw{s}$ (the other case is symmetric). Observe then that $0w'$ is right
special in $\infw{x}$, and $1w'$ is right special in $\infw{s}$. But, since they are the unique
right special words of their length, a contradiction is reached, since both words are suffixes
of $w$ as the unique right special factors of their length.

Consider the \emph{Rauzy graph $G$ of order $|w|$ of $\infw x$}.%
\footnote{The Rauzy graph, or \emph{factor graph} of order $n$ has vertex set
$V = \lang[n]{\infw{x}}$, and there is a directed edge $(u,v)$ if there exist letters
$a,b$ such that $ua = bv \in \lang[n+1]{\infw{x}}$. See \S1.3.4 and \S2.2.3 of \cite{MR1905123}
for basic properties of general Rauzy graphs and Rauzy graphs of Sturmian words, respectively.} 
Then $G$ coincides with the Rauzy graph of $\infw s$. In $\infw s$, $w$ is a bispecial factor.
Thus $G$ consists of two cycles, say with labels $x'$ and $y'$, which share the single vertex
$w$ (see \autoref{fig:RauzyGraphOfBispecial}).
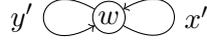
\begin{figure}
\begin{center}
	\begin{tikzpicture}
		\node[draw,circle,inner sep=1pt] (a) at (0,0) {$w$};
		\draw[->] (a) edge [loop,out=330,in=30,looseness=12] node[right] {$x'$} (a);
		\draw[->] (a) edge [loop,out=150,in=210,looseness=12] node[left] {$y'$} (a);
	\end{tikzpicture}
	\end{center}
	\caption{An illustration of the Rauzy graph of order $|xy|$ of an infinite word obtained as a product of standard pair $(x,y)$. The word $w$ is the only left (resp. right) special factor among factors of the same length.}
	\label{fig:RauzyGraphOfBispecial}
\end{figure}
Notice that $wx'$ and $wy'$ have $w$ as a suffix. Furthermore $|x'|+|y'| = |w|+2$ (the number
of edges in $G$) and, since $w \notin 0^*\cup 1^*$, we necessarily have $|x'|, |y'| \geq 2$.
It follows that $|x'| = |w| + 2 - |y'| \leq |w|$ and similarly $|y'| \leq |w|$. Therefore
$w x' = x'' w$, for some $x'' \in \words$. In particular, for $w = a_1 \cdots a_n$, we have
$w_{i} = w_{i+|x'|}$ for $i = 1$, \ldots, $n-|x'|$, so that $|x'|$ is a period of $w$. Similarly $|y'|$ is a period of 
$w$.
Since $x'$ and $y'$ begin with distinct letters, the periods must have different lengths.
By \autoref{thm:periodsOfCentral} we have, without loss of generality, $|x'|=|p|+2 = |x|$ and
$|y'| = |q|+2 = |y|$. By the above lemma, we then have
$x' = x^R$, $y' = y^R$. Hence $wx' = xw$ and $wy' = yw$. Now let $\infw y$ be any shift of
$\infw x$ beginning with $w$; clearly $\infw y$ is in the set $w\{x',y'\}^{\N} = \{x,y\}^{\N}$.

It remains to show that $\infw y$ contains both $xx$ and $yy$ in this factorization.
Consider an occurrence of $1w1$. Now since $y$ ends with $0$ and $y'$ begins with $0$, we see
that $1w1$ occurs as the central factor of $xwx' = xxw$. Thus $xx$ occurs in the factorization.
In a similar fashion, we find an occurrence of $yy$ by inspecting occurrences of $0w0$. This
concludes the proof.
\end{proof}

The following lemma can be seen as a counterpart of the previous lemma.
We need the following celebrated result of M. Morse and G. Hedlund which characterizes
ultimately periodic words
in terms of the factor complexity function.
\begin{theorem}[Morse--Hedlund]\label{thm:MorseHedlundConsecutive}
An infinite word is ultimately periodic if and only if $\compl[\infw{x}]{n} = \compl[\infw{x}]{n+1}$ for some $n \in \N$. In this case $\complfunction{\infw{x}}$ is uniformly bounded.
\end{theorem}
\begin{lemma}\label{lem:shortestUnbalancedPairBalancedSet}
Assume that an infinite binary word $\infw x$ can be expressed as a product of the standard pair $(x,y)$. Then the set of factors of length less than $|xy|$ is balanced.
\end{lemma}
\begin{proof}
The statement is true when $(x,y) \in \{(0,0^n1), (1^n0,1) \colon n\in \N\}$ by
inspection. We may thus assume that $x = p10$ and $y = q01$ for some central words $p$ and
$q$. Let $w = p10q = q01p$. It follows that $\infw x$ begins with $w$. Since $\infw{x}$ is
aperiodic, the factorization into the standard pair $(x,y)$ contains both the factors
$xy$ and $yx$. Assume that the former occurs first (the latter case is symmetric) so that
the factorization begins with $x^n xy = (p10)^nq01 = p01q(01p)^n01$ for some $n\geq 0$. It
is now evident that $\infw{x}$ begins with $w$. Furthermore, $w$ is always followed by
$x^R$ or by $y^R$. We deduce that the Rauzy graph of order $|w|$ of $\infw{x}$ is as in
\autoref{fig:RauzyGraphOfBispecial}. This implies that the number of
factors of $\infw{x}$ of length $|w|+1$ equals $|xy| = |w|+2$. Since $\infw x$ is
aperiodic, it follows by the Morse--Hedlund theorem \autoref{thm:MorseHedlundConsecutive}
that $\compl[\infw{x}]{n} = n+1$ for each $n\leq |w|+1$.

We claim that the set $X = \cup_{n \leq |w|+1} \lang[n]{\infw{x}}$ is balanced.
To see this, one
can proceed as in \cite[Thm.~2.1.5]{MR1905123}: If $X$ is not balanced, then by
\cite[Prop.~2.1.3]{MR1905123}, the set contains a palindrome $w'$ such that $0w'0$,
$1w'1 \in X$. Under our assumptions $|w'| \leq n - 1$. Since $\compl[\infw{x}]{k} = k+1$,
for each length $k < |w|+1$ there is a unique right special word $u \in X$ of length $n$.
Observe that any suffix of $u$ is also right special. Now, since $w'$ is right special, it
follows that either $0w'$ or $1w'$ is right special. Assuming that $0w'$ is right special (so
$1w'$ is not), it follows that $1w$ is always followed by $1$. Letting $v$ be a word such that
$1w'1v \in \lang[2|w'|]{\infw{x}}$. It can be shown that none of the factors of length $|0w'|$
of $1w'1v$ are right special (i.e., $0w'$ does not occur in $1w'1v$. This, further, can be
shown to imply that $\infw{x}$ is ultimately periodic. This contradiction concludes the proof.
\end{proof}

\section{Abelian closures of non-balanced words}\label{sec:non-bal}

To conclude the proof of \autoref{thm:binary}, we consider the case of non-balanced words.
\begin{proposition}\label{prop:non-balancedAbelianSubshift}
Let $\infw{x}\in \{0,1\}^{\N}$ be a uniformly recurrent, non-balanced word.
Then $\abclsr{\infw x}$ contains infinitely many minimal subshifts.
\end{proposition}

We first make a straightforward observation related to irrational letter
frequencies and morphisms.

Recall that $\Psi(u)$ is the Parikh vector of $u$.

\begin{definition}
A morphism $f \colon \{0,1\}^* \to \{0,1\}^*$ is called \emph{degenerate} if
$\Psi(f(0))$ and $\Psi(f(1))$ are linearly dependent.
Otherwise it is called non-degenerate.
\end{definition}

Notice that any erasing morphism is degenerate. On the other hand, any Sturmian morphism $\varphi$ is non-degenerate. Indeed, it can be shown, by induction on the length of a defining sequence of generators, that
$\gcd(|\varphi(0)|,|\varphi(1)|) = 1$. This suffices for non-degeneracy, as can be
established with elementary properties of integers and the fact that $|\varphi(01)|_a \geq 1$ for $a = 0,1$.

The following lemma is immediate.
\begin{lemma}
Let $f$ be a degenerate morphism. Then, for all $u$ for which $f(u) \neq \eps$, we have
$\freq[f(u)]{1} = C$ for some rational constant $C$.
\end{lemma}

On the other hand, if $f$ is non-degenerate (hence it is non-erasing), there is a one-to-one correspondence between frequencies of a word and its image. This can be seen
as follows: The \emph{adjacency matrix} $M_f$ of $f$ is defined as
\begin{equation*}
M_f = \left(\begin{matrix} |f(0)|_0 & |f(1)|_0 \\ |f(0)|_1 & |f(1)|_1 \end{matrix} \right).
\end{equation*}
It is straighforward to check that $\Psi( f(u))^{\top} = M_f \Psi(u)^{\top}$ (where $\Psi(u) = (|u|_0,|u|_1)$). Furthermore, $M_f$ is invertible if and only if $f$ is non-degenerate. Hence, for any non-degenerate morphism $f$ and an image
word $f(u)$, we can compute $\Psi(u)$ from $M_f^{-1} \Psi(f(u))^{\top}$.
We may also compute
$\left(\begin{smallmatrix} \freq[f(u)]{0} \\ \freq[f(u)]{1} \end{smallmatrix}\right)
= \frac{|u|}{|f(u)|} M_f \left(\begin{smallmatrix} \freq[u]{0} \\ \freq[u]{1} \end{smallmatrix}\right)$.

\begin{lemma}\label{lem:preimageIrrationalLetterFreq}
Let $f: \{0,1\}^* \to \{0,1\}^*$ be a morphism, and let $\infw{y} \in \{0,1\}^*$. Assume that
$f(\infw{y}) = \infw{z}$ has irrational uniform letter frequencies. Then $\infw{y}$ has irrational uniform letter frequencies. Furthermore,
if $\infw{z}$ is non-balanced, then so is $\infw{y}$.
\end{lemma}
\begin{proof}
Let $\freq[f(\infw{y})]{1} = \alpha$ be irrational. Observe that, for a degenerate morphism $f$, $f(\infw{y})$ has rational uniform letter
frequencies, as can be established by the above lemma. Hence $f$ is non-degenerate and, in particular, non-erasing.

Let $\supfreq[\infw{y}]{1} = \beta$ and
$\inffreq[\infw{y}]{1} = \beta'$ for some
numbers $\beta, \beta' \in [0,1]$. Let $(v_n)_n$ be a sequence of factors of increasing length of $\infw{y}$
testifying the former limit frequency. Then we have
$\lim_{n\to\infty} \freq[f(v_n)]{1} = \alpha$.
On the other hand
\begin{equation*}
\left( \begin{matrix} \freq[f(v_n)]{0} \\ \freq[f(v_n)]{1} \end{matrix}\right)
= \frac{|v_n|}{|f(v_n)|} M_f
\left(\begin{matrix} \freq[v_n]{0} \\ \freq[v_n]{1} \end{matrix}\right).
\end{equation*}
By a straightforward computation, we have
$\lim_{n\to \infty} \frac{|v_n|}{|f(v_n)|} = \frac{1}{|f(0)| + (|f(1)| - |f(0)|) \beta} \in (0,1)$ under our assumption on $(v_n)_n$.
Since the mapping $M_f$ is continuous, we deduce that
\begin{equation*}
\left(\begin{matrix} 1 - \alpha \\ \alpha \end{matrix}\right)
= \frac{1}{|f(0)| + (|f(1)| - |f(0)|) \beta} M_f \left(\begin{matrix} 1-\beta \\ \beta \end{matrix}\right)
\end{equation*}
It is immediate that $\beta$ is irrational. Further, we get
$M_f^{-1}\left(\begin{smallmatrix} 1 - \alpha \\ \alpha \end{smallmatrix}\right)
= \tfrac{1}{|f(0)| + (|f(1)| - |f(0)|) \beta} \left(\begin{smallmatrix} 1-\beta \\ \beta \end{smallmatrix}\right)$.
Notice now that the same computations can be performed on the sequence of factors testifying the latter limit frequency $\beta'$, only $\beta$ is replaced with $\beta'$. As a consequence, we have $h(\beta) = h(\beta')$ for the
linear fractional transformation $h(x) = \frac{x}{|f(0)| + (|f(1)| - |f(0)|) x}$. One can check that $h$ is invertible (since $f$ is non-vanishing), so we conclude that $\beta = \beta'$.
We have shown that $\infw{y}$ has irrational uniform letter frequencies.

We then show that if $\infw{y}$ is $C$-balanced for some $C$, then necessarily $\infw{z}$
is $C'$-balanced for some $C'$. Let $u$ and $v$ be equal length factors of $\infw{z}$.
There exist factors $x$, $y$ of $\infw{y}$ of minimal length for which $u$ is a factor of
$f(x)$ and $v$ is a factor of $f(y)$. As the length of $f(x)$ is bounded below by $|u|$ and above by $|u| + D$ by some constant $D$ (take, e.g., $D = 2|f(01)|$), we have
$0 \leq |f(x)|_1 - |u|_1 \leq D$. Similarly $0 \leq |f(y)|_1 - |v|_1 \leq D$.
Thus, establishing a uniform bound on $||f(x)|_1 - |f(y)|_1|$, (where $x$ and $y$ correspond to
equal length factors of $\infw{z}$) suffices to conclude the claim. Assume without loss of generality that $|x| \geq |y|$, and write $x = x'z$ with $|x'|=|y|$. Hence $\Psi(x) = \Psi(x') + \Psi(z)$. 
Observe now that
\begin{equation*}
D \geq |f(x)| - |f(y)|
=  |f(z)| + |f(x')| - |f(y)|
= |f(z)| + \langle M_f(\Psi(x') - \Psi(y))\, , (1,1) \rangle,
\end{equation*}
where $\langle \cdot\, , \cdot \rangle$ denotes the inner product.
Recall we assume that $\infw{y}$ is $C$-balanced. Hence the elements of $\Psi(x') - \Psi(y)$ have
absolute value bounded by $C$. There are finitely many such integral points, and hence we have
$\langle M_f(\Psi(x') - \Psi(y'))\,, (1,1) \rangle$ is in some bounded interval $[-D',D']$ for
some positive number $D'$ which depends on $f$ and $C$ alone. Therefore the right hand side is
bounded below by $ |z| - D'$ (since $f$ is non-erasing). We conclude that
$|z| \leq D + D'$.

Similarly we have
\begin{equation*}
|f(x)|_1 - |f(y)|_1 =  |f(z)|_1 + \langle M_f(\Psi(x') - \Psi(y)), (0,1) \rangle.
\end{equation*}
Again, the value $\langle M_f(\Psi(x') - \Psi(y)), (0,1) \rangle$ is uniformly bounded due to
the $C$-balancedness of $\infw{y}$. Here $|f(z)|_1$ is (crudely) bounded above by
$(D + D')|f(01)|$, so we conclude
that $|f(x)|_1 - |f(y)|_1 \leq (D + D') |f(01)| + D''$,
for a constant $D''$ depending solely on $f$ and $C$. This concludes the proof.
\end{proof}

We have now developed sufficient tools to prove
\autoref{prop:non-balancedAbelianSubshift}.

In what follows, $\infw x$ is a uniformly recurrent, non-balanced binary word having
irrational letter frequencies. We may assume that $\freq[\infw{x}]{1} < 1/2$ and, without
loss of generality, all $1$s are isolated in $\infw{x}$. Otherwise, by
\autoref{prop:finitelyManyIterationsAll1sIsolated}, there exists a word in
$\abclsr{\infw{x}}$ with this property, and we may argue about its abelian closure.

Our aim is to define, for each $n\geq 0$, a uniformly recurrent word $\infw x_n$
in $\abclsr{\infw x}$. These words define pairwise distinct shift orbit closures
in $\abclsr{\infw x}$, which suffices for the claim.

For the construction, we actually define three sequences of words,
$(\infw{x}_n)_{n\geq 0}$, $(\infw{y}_n)_{n\geq 1}$, and $(\infw{z}_n)_{n\geq 0}$, as well as two sequences $(\psi_n)_{n\geq 0}$ and
$(\varphi_n)_{n\geq 1}$ of standard morphisms recursively. This will help to keep track of
the properties we need for the conclusion.
The entities satisfy the following properties for all $n\geq 0$.
\begin{enumerate}[topsep=2pt,itemsep=-2pt]
\item
\label{it:ynp1Props}
$\infw{y}_{n+1}$ is non-balanced, uniformly recurrent, has $\freq[\infw{y}_n]{1}$
irrational and less than $1/2$, and it contains both $00$ and $11$.

\item
\label{it:phipsidef}
$\psi_{n+1} = \psi_n \circ \varphi_{n+1}$ and $\varphi_{n+1}$ is a non-trivial
(meaning $\varphi_{n+1}(01)\geq 3$)
standard morphism.

\item
\label{it:znIsPhinp1Ynp1}
$\infw{z}_n = \varphi_{n+1}(\infw{y}_{n+1})$.

\item
\label{it:znNonbalancedEtc}
$\infw{z}_n$ is non-balanced, uniformly recurrent, has $\freq[\infw{z}_n]{1}$
irrational and less than $1/2$. Further, all $1$s are isolated.

\item
\label{it:znp1InASsYnp1}
$\infw{z}_{n+1} \in \abclsr{\infw{y}_{n+1}}$.

\item
\label{it:imageZnp1InASsZn}
$\varphi_{n+1}(\infw{z}_{n+1}) \in \abclsr{\infw{z}_n}$.

\item
\label{it:xnIsPsiZn}
$\psi_{n}(\infw{z}_n) = \infw{x}_n$.

\item
\label{it:xnInASsX}
$\infw{x}_n \in \abclsr{\infw{x}}$.
\end{enumerate}

First we set $\infw{x}_0 = \infw{z}_0 = \infw{x}$, and $\psi_0 = id$. The above list of
properties concerning these entities hold immediately. The definitions of $\infw{y}_1$ and
$\varphi_1$ are evident from the construction that follows. The construction is depicted
in \autoref{fig:constructionOfInfinitelyManyMinimalSubshifts}.

\begin{figure}
\begin{center}
\begin{tikzpicture}[x=2.5cm,y=1.2cm]

\node (x0) at (0,0) {$\infw{x}_0=\infw{z}_0$};

\node (y1) at (0,1) {$\infw{y}_1$};

\draw[->,dashed] (x0) edge node[left] {\footnotesize{$\varphi_1^{-1}$}} (y1);

\node (x1) at (1,0) {$\infw{x}_1$};

\node (z1) at (1,1) {$\infw{z}_1$};

\node (y2) at (1,2) {$\infw{y}_2$};

\draw[->,dashed] (y1) edge node[above] {\footnotesize{$F^{n_1}$}} (z1);

\draw[->,dashed] (z1) edge node[left] {\footnotesize{$\psi_1$}} (x1);

\draw[->,dashed] (z1) edge node[left] {\footnotesize{$\varphi_2^{-1}$}} (y2);

\node (x2) at (2,0) {$\infw{x}_2$};

\node (z2) at (2,2) {$\infw{z}_2$};

\node (y3) at (2,3) {$\infw{y}_3$};

\draw[->,dashed] (y2) edge node[above] {\footnotesize{$F^{n_2}$}} (z2);

\draw[->,dashed] (z2) edge node[left] {\footnotesize{$\psi_2$}} (x2);

\draw[->,dashed] (z2) edge node[left] {\footnotesize{$\varphi_3^{-1}$}} (y3);

\node (x3) at (3,0) {$\infw{x}_3$};

\node (z3) at (3,3) {$\infw{z}_3$};

\node (y4) at (3,4) {$\infw{y}_4$};

\draw[->,dashed] (y3) edge node[above] {\footnotesize{$F^{n_3}$}} (z3);

\draw[->,dashed] (z3) edge node[left] {\footnotesize{$\psi_3$}} (x3);

\draw[->,dashed] (z3) edge node[left] {\footnotesize{$\varphi_4^{-1}$}} (y4);

\node (a) at (3.7,0) {$\cdots$};

\node (b) at (3.7,3.8) {\reflectbox{$\ddots$}};

\node (b) at (3.7,4.8) {\reflectbox{$\ddots$}};

\node (xm) at (5,0) {$\infw{x}_{n}$};

\node (zm) at (5,5) {$\infw{z}_n$};

\node (ym) at (5,6) {$\infw{y}_{n+1}$};

\node (em) at (4,5) {$\infw{y_n}$};

\draw[->,dashed] (em) {} edge node[above] {\footnotesize{$F^{m_n}$}} (zm);

\draw[->,dashed] (zm) edge node[left] {\footnotesize{$\psi_n$}} (xm);

\draw[->,dashed] (zm) edge node[left] {\footnotesize{$\varphi_{n+1}^{-1}$}} (ym);

\end{tikzpicture}
\end{center}
\caption{A diagram depicting the relationship of the families of infinite words
$(\infw{x}_n)_{n\geq 0}$, $(\infw{y}_{n\geq 1})$, and $(\infw{z}_{n\geq 0})$, and standard
morphisms $(\psi_n)_{n\geq 0}$ and $(\varphi_n)_{n\geq 1}$. Here $F$ is the operation
defined in \autoref{def:F-operationCA}, and $n_m$ is the integer $n$ alluded to in
\autoref{prop:finitelyManyIterationsAll1sIsolated}.
}
\label{fig:constructionOfInfinitelyManyMinimalSubshifts}
\end{figure}
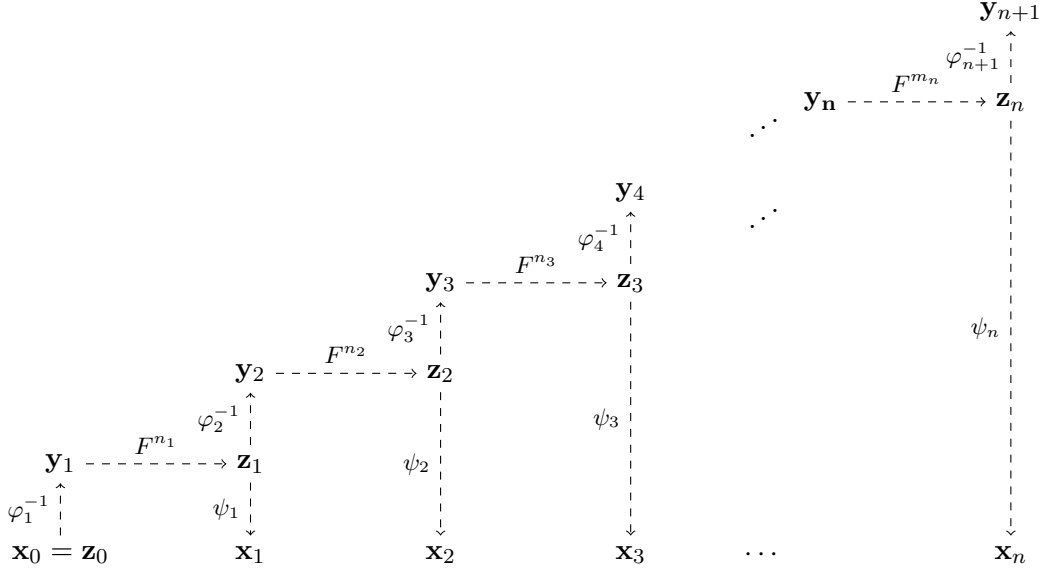

Assume then that $\infw{x}_n$, $\infw{z}_n$, $\psi_n$ are defined and satisfy the above
properties. We shall construct $\infw{y}_{n+1}$ and $\varphi_{n+1}$ from these entities,
so the knowledge of $\infw{y}_n$ and $\varphi_n$ are not needed. Let us do this first.
Since $\infw{z}_n$ is non-balanced and uniformly recurrent, by
\autoref{lem:factorizationOfShift} there exists a standard pair $(x,y)$ such that a shift
of $\infw{z}_n$ is a product of the words $x,y$ and contains both $xx$ and $yy$ in the
factorization. Let us denote this shift by $\infw{z}'$. Consider the morphism $\varphi$
defined by $0 \mapsto x$ and $1 \mapsto y$. As $\infw{z}'$ is a product of $x$ and $y$,
there exists a word $\infw{y}$ such that $\varphi(\infw{y}) = \infw{z}'$. Observe now that
$\infw{y}$ contains both $00$ and $11$. Because $\infw{z}'$ has irrational uniform letter
frequencies and is non-balanced, $\infw{y}$ shares these properties by
\autoref{lem:preimageIrrationalLetterFreq}. We may assume that $\freq[\infw{y}]{1} < 1/2$,
otherwise we replace $\varphi$ by $\varphi \circ E$. Now set $\infw{y}_{n+1} = \infw{y}$
and $\varphi_{n+1} = \varphi$, which is a standard morphism by
\autoref{thm:AdeLucaStandarMorphism}. It is also non-trivial, since $\infw{y}_{n+1}$
contains both $00$ and $11$ while $\infw{z}_n$ does not. We have thus established
items \ref{it:ynp1Props}, \ref{it:phipsidef}, and \ref{it:znIsPhinp1Ynp1} in the above
list. For the remainder of the construction, we omit the subscript from $\infw{y}_{n+1}$
for the sake of readability.

We then define $\infw{z}_{n+1}$. By \autoref{cor:AbelianSsNonBalancedContainsIsolated1}
there exists a non-balanced word in $\abclsr{\infw{y}}$ in which all $1$s are
isolated. We set $\infw{z}_{n+1}$ to be such a word. Now $\infw{z}_{n+1}$ is uniformly
recurrent, has irrational uniform letter frequencies with $\freq{1} < 1/2$ and is
non-balanced, as $\infw{y}$ has these properties. These observations establish
\autoref{it:znNonbalancedEtc} and \autoref{it:znp1InASsYnp1}. Further, since
$\infw{z}_{n+1} \in \abclsr{\infw{y}_{n+1}}$, by
\autoref{prop:AbSsClosedUnderSturmian} it follows that
$\varphi_{n+1}(\infw{z}_{n+1}) \in \abclsr{\varphi_{n+1}(\infw{y})}
= \abclsr{\infw{z}_{n}}$.
This establishes \autoref{it:imageZnp1InASsZn}. We finally define
$\infw{x}_{n+1} = \psi_{n+1}(\infw{z}_{n+1})$ in accordance with \autoref{it:xnIsPsiZn}.

Let us show that $\infw{x}_{n+1}$ satisfies \autoref{it:xnInASsX}. By
\autoref{prop:AbSsClosedUnderSturmian}, we have that
$\varphi_{n+1}(\infw{z}_{n+1}) \in \abclsr{\varphi_{n+1}(\infw{y})}
= \abclsr{\infw{z}_n}$.
Applying \autoref{prop:AbSsClosedUnderSturmian} again, this time to
$\varphi_{n+1}(\infw{z})$ and $\infw{z}_n$ with $\psi_n$, we find
\begin{equation*}
\infw{x}_{n+1} = \psi_{n+1}(\infw{z}_{n+1})
= \psi_n \circ \varphi_{n+1}(\infw{z}_{n+1})
\in \abclsr{\psi_n(\infw{z}_n)}
= \abclsr{\infw{x}_n}
\subseteq \abclsr{\infw{x}},
\end{equation*}
as $\infw{x}_n$ satisfies \autoref{it:xnInASsX} was assumed.

The following lemma combined with \autoref{it:xnInASsX} proves
\autoref{prop:non-balancedAbelianSubshift} immediately.
\begin{lemma}\label{it:xnXkDisjoint}
For all $m \neq n$, we have
$\soc{\infw{x}_n} \cap \soc{\infw{x}_m} = \emptyset$.
\end{lemma}
\begin{proof}
We show that the words have distinct factor sets. This suffices for the proof, since the words $\infw{x}_n$ are uniformly recurrent (by items \ref{it:xnIsPsiZn} and \ref{it:znNonbalancedEtc}). Consider a fixed index $n \geq 0$. Combining items \ref{it:xnIsPsiZn},
\ref{it:znIsPhinp1Ynp1}, and \ref{it:phipsidef}, we have
$\infw{x}_n = \psi_n(\infw{z}_n)
= \psi_n(\varphi_{n+1}(\infw{y}_{n+1}))
= \psi_{n+1}(\infw{y}_{n+1})$.
By \autoref{it:ynp1Props}, $\infw{y}_{n+1}$ contains both $00$ and
$11$. Thus $\infw{x}_n$ is a product of the factors $\psi_{n+1}(0) = x_n$ and
$\psi_{n+1}(1) = y_n$, and both $x_n x_n$ and $y_n y_n$ occur in $\infw{x}_{n}$. By
\autoref{thm:AdeLucaStandarMorphism} $(x_n,y_n)$ (or $(y_n,x_n)$) is a standard pair.
Further, by \autoref{lem:factorizationOfShift} the shortest unbalanced pair of factors has
length $|x_n y_n|$, and by \autoref{lem:shortestUnbalancedPairBalancedSet}, the factors of length
less than $|x_n y_n|$ form a balanced set.

To conclude the proof, it suffices to show that $|x_{n+1} y_{n+1}| > |x_n y_n|$ for all
$n \geq 0$. By \autoref{it:phipsidef}, we find
$\psi_{n+1} = \psi_n \circ \varphi_{n+1}$. We note that $|x_{n+1}y_{n+1}|_0$,
$|x_{n+1}y_{n+1}|_1 \geq 1$. Furthermore, one of these values is at least $2$ since
$\varphi_{n+1}$ is non-trivial. We thus have
\begin{equation*}
|x_{n+1}y_{n+1}| = |\psi_{n}(\varphi_{n+1}(01))|
= |x_{n+1}y_{n+1}|_0 \cdot |x_n| + |x_{n+1}y_{n+1}|_1 \cdot |y_n| > |x_n y_n|.
\end{equation*}
This concludes the proof.
\end{proof}

\section{Some remarks on alternative approaches}\label{sec:alt}
In this section we discuss some alternative approaches to the results presented in the
preceding sections. We first show that a large family of words with irrational letter frequencies contain uncountably many
minimal subshifts in their abelian closures. 
We then 
show that trying to apply to  \autoref{prop:non-balancedAbelianSubshift} the approaches from the proofs of the other cases does not give the result in the full generality, although shows stronger results in particular cases, as well as demonstrates some new phenomena. 

\subsection{On abelian closures with uncountably many minimal subshifts}

In this subsection we show that certain words $\infw{x}$ with irrational letter frequencies
have uncountably many minimal subshifts in their abelian closures. We apply methods from the
proof of \autoref{prop:rationalFreqUncountable}. Notice however, that this does not give a stronger version of Propositions \ref{Prop:C-bal} or \ref{prop:non-balancedAbelianSubshift}  in full generality.

\begin{proposition}\label{prop:wideStripUncountablyMany}
Let $\infw{x}$ be uniformly recurrent with $\freq[\infw{x}]{1} = \alpha$ with $\alpha$
irrational. Assume further that, for each $n \geq n_0$ for some $n_0 \in \N$, it
contains factors of length $n$ with one having weight $\lceil n\alpha \rceil + 1$ and another
having weight $\lfloor n\alpha \rfloor - 1$. Then $\abclsr{\infw{x}}$ contains uncountably
many minimal subshifts.
\end{proposition}

We may assume without loss of generality that $\alpha < 1/2$. We construct a family of words in $\abclsr{\infw{x}}$ as follows.

Let $\infw{c}$ be the characteristic Sturmian word of slope $\alpha$. Let $(a_n)_{n\geq 1}$ be
the corresponding directive sequence, and $(S_n)_{n \geq -1}$ the standard sequence. Recall
that $S_n = S_{n-1}^{a_n} S_{n-2}$ for each $n\geq 1$. We shall consider a modification of this sequence as follows.

Notice that $\infw{c} \in \abclsr{\infw{x}}$ by the \hyperref[lem:corridorLemma]{Corridor Lemma}. We aim to "spread" the
graph of $\infw{c}$ around the line $y = \alpha x$ so that the obtained word is also in
$\abclsr{\infw{x}}$. This is the part where we need extra room around the slope $\alpha x$,
which is granted by the assumptions. To this end, let $k \geq 3$ be such that $|S_{k+1}| = |S_k^{a_{k+1}}S_{k-1}| > n_0$ (notice that $|S_{k-1}|\geq 2$ for $k \geq 3$). Now $\infw{c}$
is a product of the words $S_k$ and
$S_{k-1}$: let us write
\begin{equation}\label{eq:sturmianProduct}
\infw{c} = \prod_{i=0}^{\infty} S_k^{n_i} S_{k-1}.
\end{equation}
Here $n_i$ is one of the two numbers $a_{k+1}$, $a_{k+1}+1$, for each $i\geq 0$.
Let $\mathfrak{F}$ denote the operation which flips the last two letters of a given word (of length at least two).
\begin{claim}\label{claim:0-1-flippings}
Let $(b_i)_{i\geq 0}$ be a $0$-$1$-sequence. Then the word $\infw{x}'$ defined by
\begin{equation}\label{eq:sturmianProductMod}
\infw{x}' = \prod_{i=0}^{\infty} S_k^{n_i} \mathfrak{F}^{b_i}(S_{k-1})
\end{equation}
is in $\abclsr{\infw{x}}$.
\end{claim}
\begin{proof}
Notice that $\{S_k,S_{k-1}\}$ is an unordered standard pair. Since $S_k S_{k-1}$ and $S_{k-1} S_k$
differ in only the last two letters (\cite[Prop.~2.2.2]{MR1905123}) we have
$S_{k} \mathfrak{F}(S_{k-1}) = S_{k-1}S_k$. It is then evident that any word of the form
\eqref{eq:sturmianProductMod} is a product of the standard pair $\{S_k^{a_{k+1}-1}S_{k-1},S_{k}\}$.
By \autoref{lem:shortestUnbalancedPairBalancedSet}, the set of factors of length
less than $|S_k^{a_{k+1}} S_{k-1}| = |S_{k+1}|$ forms a balanced set.

The rest of the proof is similar to that of \autoref{lem:rational}.
Assume for simplicity that $S_{k-1}$ ends with $01$ (equivalently, $k$ is odd, $k\geq 3$). The other case is totally symmetric.
Take a factor $u$ of
$\infw{x}'$ and express it as $u = s P p$, where $P$ is a finite sub-product of \eqref{eq:sturmianProductMod} and $s$ (resp., $p$) is a proper suffix (resp., prefix) of the previous (resp., following) term. Consider the corresponding factor $s'P'p'$ from $\infw{s}$.
We have $|u|_1 - |v|_1 = |s|_1 - |s'|_1 + |p|_1 - |p'|_1$. Notice that $||s|_1 - |s'|_1| \leq 1$, and $1$ is only attained with $s = 0$ (and thus $s' = 1$). Similarly $||p|_1 - |p'|_1| \leq 1$ and $1$ is attained only when $p = S_k^{n_i} \mathfrak{F}(S_{k-1})0^{-1}$. If both happen
simultaneously, then $|u|_1 - |v|_1 = 0$. Consequently, $||u|_1 - |v|_1| \leq 1$. Hence $\lfloor n\alpha \rfloor - 1 \leq |u|_1 \leq \lceil n\alpha \rceil + 1$.
By \autoref{lem:corridorLemma} and our assumptions on $\infw{x}$, $\infw{x}' \in \abclsr{\infw{x}}$.
\end{proof}

We are going to prove that there are uncountably many $0$-$1$-sequences $(b_n)_{n=0}^{\infty}$
such that the corresponding words of the form \eqref{eq:sturmianProductMod} have distinct sets
of factors. One of the ways to do this is using yet another characterization of Sturmian words
via rotations.

We identify the interval $[0,1)$ with the unit circle $\mathbb T$
(the point $1$ is identified with point $0$). For points
$x,y\in\mathbb T$, we let ${I}(x,y)$ denote the half-open interval
on $\mathbb T$ starting from $x$ and ending at
$y$ in counter--clockwise direction (in most of the arguments it will not be important which
endpoint is in the interval).
Let $\alpha \in \mathbb{T}$ be irrational and let $\rho\in \mathbb T$.
The map $R_{\alpha}:\mathbb T \to \mathbb T$, $x \mapsto
\{x+\alpha\}$, where $\{x\} = x-\lfloor x \rfloor$ denotes the
fractional part of $x\in\R$, defines a (counter-clockwise) rotation on $\mathbb T$.
Partition $\mathbb T$ into two half-open intervals $I_0 = I(0,1-\alpha)$ and
$I_1 = I(1-\alpha,1)$ (so the endpoints $1-\alpha$ and $0 = 1$ are in different partitions), and
define the coding $\nu:\mathbb T\to \{0,1\}$, $x\mapsto i$ if $x\in I_i$,
$i=0,1$. The \emph{rotation word} ${\infw s}_{\alpha,\rho}$ of \emph{slope} $\alpha$ and
\emph{intercept} $\rho$ is the word $a_0a_1\cdots \in \{0,1\}^{\N}$ defined by
$a_n = \nu(R_{\alpha}^n(\rho))$ for all
$n\in\N$.

Note that $00$ occurs in $\infw s_{\alpha,\rho}$ if
and only if $\alpha < 1/2$. Clearly, $\infw s_{\alpha,\rho}$ is aperiodic as $\alpha$ is
irrational. Each aperiodic rotation word is a Sturmian word and vice versa (regardless of the choice of whether $1 \in I_0$ or $0\in I_0$). For each length
$n$, one can partition the interval $[0,1)$ into $n+1$ subintervals, each of which corresponds
to a factor of the Sturmian word. More precisely, we can find when $v=b_1\cdots b_n$ occurs in
$\infw{s}$ at position $i$:
$$v = s_i s_{i+1} \cdots s_{i+n-1}  \Leftrightarrow  R_{\alpha}^i(\rho) \in I_v,$$
where
$$I_v = I_{b_1} \cap R_{\alpha}^{-1}(I_{b_2}) \cap \cdots \cap R_{\alpha}^{-n+1}(I_{b_{n}}).$$

\begin{example}
In \autoref{fig:SturmianRotation} we have an example of a rotation system corresponding
to a Sturmian word of slope $\alpha$. Here we assume that $3\alpha < 1 < 4\alpha$.
The intervals defined by the points $\{-i\alpha\}$, $i=0,\ldots,4$, define the factors
of length $4$ as follows:
$I(0,\{ -3\alpha \})$ corresponds to the factor $0^31$,
$I(\{ -3\alpha \},\{ -2\alpha \})$ corresponds to $0^210$,
$I(\{-2\alpha\},\{-\alpha\})$ corresponds to $0100$,
$I(\{ -\alpha \},\{ -4\alpha \})$ corresponds to $10^3$, and
$I(\{ -4\alpha \},1)$ corresponds to $1001$.
\end{example}
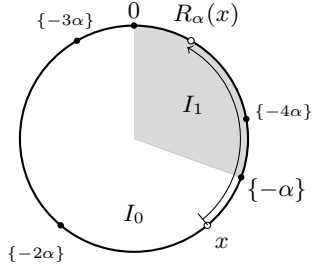
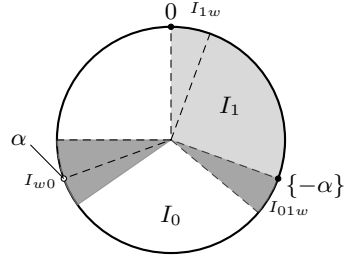
\begin{figure}
\centering
\begin{subfigure}{0.45\textwidth}
\centering
\begin{tikzpicture}
    \filldraw[gray,opacity=.3] (0,0) -- (-20:1.5) arc (-20:90:1.5);
    \draw[thick] (0,0) circle(1.5);
    
    \node at (270:1) {\footnotesize{$I_0$}};
    \node at (30:.9) {\footnotesize{$I_1$}};
    
    \filldraw[white,draw=black] (310:1.5) circle(1.2pt);
    \node at (310:1.8) {\footnotesize{$x$}};

    \filldraw[white,draw=black] ({310+110}:1.5) circle(1.2pt);
    \node at (310+110:1.9) {\footnotesize{$R_{\alpha}(x)$}};
    \draw[|->] (310:1.40) arc (310:310+110:1.40);

    \filldraw (90:1.5) circle(1pt);
    \node at (90:1.7) {\footnotesize{$0$}};
    
    \filldraw (340:1.5) circle(1pt);
    \node at (340:2) {\footnotesize{$\{-\alpha\}$}};

    \node at ({340-110}:2) {\tiny{$\{-2\alpha\}$}};
    \filldraw ({340-110}:1.5) circle(1pt);

    \node at ({340-2*110}:1.8) {\tiny{$\{-3\alpha\}$}};
    \filldraw ({340-2*110}:1.5) circle(1pt);
    
        \node at ({340-3*110}:2) {\tiny{$\{-4\alpha\}$}};
    \filldraw ({340-3*110}:1.5) circle(1pt);
    

\end{tikzpicture}
\caption{Illustration of a rotation word. The coding of the orbit of the point $x$
under $\nu$ begins with $01$.}
\label{fig:SturmianRotation}
\end{subfigure}
\quad
\begin{subfigure}{0.45\textwidth}
\centering
\begin{tikzpicture}
    \filldraw[gray,opacity=.3] (0,0) -- (-20:1.5) arc (-20:90:1.5);
    \draw[thick] (0,0) circle(1.5);

\filldraw[gray, opacity=.7] (0,0) -- (180:1.5) arc (180:215:1.5) -- (0,0);

    \draw[densely dashed] (0,0) -- (180:1.5) arc (180:200:1.5) -- (0,0);
    \node at (197.5:1.8) {\tiny{$I_{w0}$}};

    \draw[gray,draw = black,densely dashed] (0,0) -- (70:1.5) arc (70:90:1.5) -- (0,0);
    \node at (75:1.8) {\tiny{$I_{1w}$}};
    
    \filldraw[gray,draw = black, densely dashed,opacity=.7] (0,0) -- (320:1.5) arc (320:340:1.5) -- (0,0);
    \node at (330:1.8) {\tiny{$I_{01w}$}};
    
    \node at (270:1) {\footnotesize{$I_0$}};
     \node at (30:.9) {\footnotesize{$I_1$}};   


	 \filldraw (90:1.5) circle(1pt);
    \node at (90:1.7) {\footnotesize{$0$}};
    
    \filldraw (340:1.5) circle(1pt);
    \node at (342:2) {\footnotesize{$\{-\alpha\}$}};
    
    \filldraw[white,draw=black] (200:1.5) circle(1pt);
    \node (b) at (180:2) {\footnotesize{$\alpha$}};

    \draw (182:1.85)--(200:1.5);

\end{tikzpicture}
\caption{Illustration of the intervals corresponding to the factors $w0$ (dark sector), $1w$,
and $01w$ of
$\infw{c}$.}
\label{fig:SturmianMod}
\end{subfigure}
\caption{An illustration of a system of codings of rotations.}
\end{figure}

Observe the special role played by $\infw{s}_{\alpha,\alpha}=\infw{c}$:
both $01\infw{s}$ and $10\infw{c} \in \soc{\infw c}$ for any $\alpha \in (0,1)$. This follows
from the fact that we may choose first that $0 \in I_0$ in which case $1\in I_1$. Then
$\infw{s}_{-\alpha,\alpha} = 01\infw{c}$. The choice $1\in I_0$ gives
$\infw{s}_{-\alpha,\alpha} = 10\infw{c}$.

Let us assume for simplicity that $S_{k-1}$ ends with $01$ for the remainder of this subsection.
We claim that each occurrence of $01w$ in \eqref{eq:sturmianProduct} (with $w01 = S_{k+1}$) starts from the
second to last letter of each $S_{k-1}$ in the factorization \eqref{eq:sturmianProduct}.
Indeed, by \autoref{thm:periodsOfCentral} $w$ has two periods: $|S_k|$ 
and $|S_k^{a_{k+1}-1} S_{k-1}|$ and further \autoref{lem:centralStandardPeriods} tells how consecutive occurrences
of $w$ appear. So in \eqref{eq:sturmianProduct} occurrences of $w$ correspond to prefixes of each block $S_{k}^{n_i}S_{k-1}$ (preceded by $01$) and to factors starting from position
$|S_{k}|$ (also $2|S_k|$ if $n_i = a_{k+1}+1$) of the factor $S_{k}^{n_i} S_{k-1} \cdot S_{k}$. The latter occurrences of $w$ are preceded by $10$.

\begin{claim}\label{prop:uncountably_manyMod}
There are uncountably many $0$-$1$-sequences $(b_n)_{n=0}^{\infty}$ so that the sequences of the form
\eqref{eq:sturmianProductMod} have distinct sets of factors.
\end{claim}
\begin{proof}
Take the Sturmian word $\infw{c}$ as defined above. Consider now the interval $I_{01w}$ corresponding to the factor $01w$ in $\infw{c}$.
We have $I_{01w} = I(a,1-\alpha)$ for some $a < 1-\alpha$: Since both
$0w0$ and $1w0$ occur in $\infw{c}$, it follows that $I_{w0}$ contains the point $\alpha$.
Now $I_1 \cap R_{\alpha}^{-1}(I_{w0}) = I_{1w}$ and is of the form $I(a+\alpha,1)$ for some $a < 1-\alpha$. Hence $I_{01w} = R_{\alpha}^{-1}(I_{1w}) = I(a,1-\alpha)$.
Observe now that each time the orbit of $\infw{c}$ hits the interval
$I_{01w}$, it synchronizes with the factorization
\eqref{eq:sturmianProduct} as describe in the above discussion. Let us modify the 
coding $\nu$ to $\nu'$ in such a way that allows to flip the last two letters of $S_{k-1}$ to obtain a word of the form \eqref{eq:sturmianProductMod}. Take a subinterval
$J$ of $I_{01w}$ that does not have $1-\alpha$ as an endpoint. Then
$J' = R(J)$ is a subinterval of $1w$ that does not have $1$ as an endpoint.
Define $\nu' \colon \mathbb{T} \to \{0,1\}$ by $\nu'(x) = 1$ if
$x \in J$ or if $x \in I_1 \setminus J'$. Similarly $\nu'(x) = 0$ if 
$x\in J'$ or if $x\in I_0 \setminus J$. So the coding $\nu'$ partitions
the torus into six subintervals:
letting $J = I(a,b)$, the intervals are in anti-clockwise order

\begin{itemize}[topsep=2pt,itemsep=3pt]
\item $I(0,a)$ ($\mapsto 0$ under $\nu'$),

\item $I(a,b) = J$ ($\mapsto 1$),

\item $I(b,1-\alpha)$ ($\mapsto 0$),

\item $I(1-\alpha, a + \alpha)$ ($\mapsto 1$),

\item $I(a + \alpha, b + \alpha) = J$ ($\mapsto 0$),
 
\item $I(b + \alpha ,1)$ ($\mapsto 1$).
\end{itemize}

Code now the orbit of the point $\alpha$ under $\nu'$ to obtain an infinite word $\infw{t}$.
This coding acts the same as the coding under $\nu'$, except when the 
orbit hits a point in $J$. The factor starting from this interval is 
$10w$ (as opposed to $01w$ in $\infw{c}$). It is now immediate that
$\infw{t}$ is of the form \eqref{eq:sturmianProductMod}.

We claim that varying the length of $J$ we get uncountably many minimal subshifts.
We use the notion of factor frequency, generalizing letter frequencies. The \emph{uniform
frequency of a factor $z$ of $\infw{y}$} is defined as the limit
$\freq[\infw{y}]{z} = \lim_{N\to \infty}\frac{|v_N|_z}{N}$ when it exists, uniformly over $(v_n)_{n=0}^{\infty}$ being any sequence of factors of $\infw{y}$
with $|v_N| = N$, where $|v|_z$ is the number of occurrences of $z$ as a factor in $v$. It can
be seen that for any Sturmian word $\infw{y}$ and for any finite word $z$, the frequency
$\freq[\infw{y}]{z}$ exists and is equal to the length of the corresponding interval on the
torus \cite[\S2.2.3]{MR1905123}. Using precisely the same argument, one can see that the frequency
of any factor of $\infw{t}$ exists and equals the length of the corresponding interval/set
on the corresponding torus. In particular, the frequency of the factor $1w1$ exists. An occurrence of $1w1$ corresponds exactly to an
occurrence of $S_{k-1} \cdot S_{k}^{a_{k+1}} \mathcal{F}(S_{k-1})$ in
\eqref{eq:sturmianProductMod} (choosing $J$ appropriately gives occurrences of this form). Indeed, $1w1$ does not occur in $\infw{c}$
so it must overlap with at least one of the last two letters of an occurrence of $\mathfrak{F}(S_{k-1})$. If it overlaps both letters, then
it is a factor of
$S_{k}^{a_{k+1}}\mathfrak{F}(S_{k-1})S_{k}^{a_{k+1}}S_{k-1}(01)^{-1} = S_{k}^{a_{k+1}-1} S_{k-1}S_{k}^{a_{k+1}+1}S_{k-1}(01)^{-1}$ which is a factor of $\infw{c}$ contradicting balancedness. So it overlaps only one of the last two letters, so we deduce
that it occurs as a prefix of $1S_{k}^{a_{k+1}}\mathfrak{F}(S_{k-1})$. Since
both $S_k$ and $\mathfrak{F}(S_{k-1})$ end with $10$ we deduce that the
prefix $1$ ends an occurrence of $S_{k-1}$ in the factorization \eqref{eq:sturmianProductMod} as claimed.

Now the frequency of $1w1$ in $\infw{t}$ equals the size of the corresponding set on the torus, which can be varied continuously.
Choosing $J$ appropriately gives words with different frequencies of
$1w1$, which must have distinct sets of factors. The claim follows.
\end{proof}

\subsection{Additional remarks on the structure of the abelian closures of non-balanced binary words}\label{subsec:C-squeezingRevisited}


In this subsection we give a geometric proof of a weaker version of \autoref{prop:non-balancedAbelianSubshift} and  show that the abelian shift orbit closure of a uniformly recurrent word can contain non uniformly recurrent words of a quite complicated structure:

\begin{proposition} \label{prop:unbal}
Let $\infw{x}$ be a binary uniformly recurrent word which is not Sturmian. Suppose in addition
that $\infw{x}$ is non-balanced and that the frequency $\alpha$ of 1 exists and it is irrational. Then $\abclsr{\infw{x}}$ contains infinitely many non uniformly recurrent words  with distinct languages, such that none of their tails is uniformly recurrent.
\end{proposition}

We remark that this proposition does not guarantee infinitely many minimal subshifts in the abelian closure, since these words with distinct languages can have the same languages of uniformly recurrent points in their shift orbit closure.

To prove this proposition, we again make use of graphs of words, as well as squeezing operations. First we need the following lemma, which is a slight modification of item 3 of Theorem 3 from \cite{AP16}:
\begin{lemma}\label{lem:line}
Let $\infw{w}$ be a binary uniformly recurrent word with frequency of $1$ equal to $\alpha$. Then there exists $\infw{u} \in \soc{\infw{w}}$ such that $g_{\infw{u}}$ intersects the line $y=\alpha x$ infinitely many times.
\end{lemma}

\begin{proof}
In the proof we use the notion of a return word. For $u\in \lang{\infw{w}}$, let $n_1<n_2<\dots$ be all integers $n_i$ such that $u=w_{n_i}\dots w_{n_{i}+|u|-1}$. Then the word $w_{n_i}\dots w_{n_{i+1}-1}$ is a \emph{first return word} (or briefly \emph{first return}) of $u$ in $\infw{w}$ \cite{durand,hz,PUZYNINA2013390}. We can also consider a second (third, etc.) return as a factor having exactly two occurrences of $u$, one of them being a prefix, and ending just before the next occurrence of $u$.

We now build a word $\infw{u}$ as a limit of factors of $\infw{w}$. Start with any factor $u_1$ of $\infw{w}$, e.g. with a letter. Without loss of generality assume that
$\freq[u_1]{0} \geq \freq[\infw{w}]{0} = \rho_0$. Consider the factorization of $\infw{w}$ into returns
to $u_1$: $\infw{w} = v^{(1)}_1 v^{(1)}_2 \dots v^{(1)}_i \dots$, so that $v^{(1)}_i$ is a return to $u_1$ for $i>1$. We assume that all $v^{(1)}_i$ are longer than $u_1$, taking second returns (or third returns etc. if necessary). Then there exists $i_1>1$ satisfying $\freq[v_{i_1}^{(1)}]{0} \geq \rho_0$. Suppose the converse holds, i.e., for all $i>1$ $\freq[v_i^{(1)}]{0} < \rho_0$. Due to uniform recurrence, the lengths of $v_i^{(1)}$ are uniformly bounded, 
and hence $\freq[\infw{w}]{0} < \rho_0$, a contradiction. Take $u_2=v^{(1)}_{i_1}$, so $u_1$ is a prefix of $u_2$. Now consider a factorization of $\infw{w}$ into returns to $u_2$:
$\infw{w} = v^{(2)}_1 v^{(2)}_2 \cdots v^{(2)}_i \dots$. Then there exists $i_2>1$ satisfying $\freq[v_{i_2}^{(2)}]{0} \leq \rho_0$; take $u_3=v^{(2)}_{i_2}$. Continuing this line of reasoning to infinity, we build a word $\infw{u} = \lim_{n\to\infty}u_i \in \soc{\infw{w}}$, such that $\freq[u_{2i}]{0} \geq \rho_0$, $\freq[u_{2i+1}]{0} \leq \rho_0$. So, the graph of $\infw{u}$ intersects the line $y=\alpha x$ infinitely many times as was claimed. 
\end{proof}

\begin{example}
For $\infw{w} = \infw{s}_{0,\alpha}$, $g_{\infw{w}}$ does not intersect $y=\alpha x$ infinitely many times. Taking for example  $\infw{u} = \infw{s}_{\alpha,\alpha}$, we already have infinitely many intersections.
\end{example}

Let $\infw{u} \in \soc{\infw{w}}$ be a word satisfying \autoref{lem:line} and let $C \in \mathbb{R}, C>1$. We define an operation of \emph{$C$-squeezing} of $\infw{u}$, $\infw{u}'=s_C(\infw{u})$ as follows. For each $i$ such that $g_{\infw{u}}(i)>\alpha i + C$ and $u_{i-1}=1$, $u_i=0$, we define $u'_{i-1}=0$, $u'_i=1$. Symmetrically, if $g_{\infw{u}}(i) < \alpha i - C$  and $u_{i-1}=0$, $u_i=1$, we define $u'_{i-1}=1$, $u'_i=0$. In these cases we say that we have a switch at position $i$. Otherwise we define $u'_i=u_i$. Informally, this means that if there is a piece of graph outside the stripe between the lines $y=\alpha x - C$ and $y=\alpha x + C$, we make local changes in this piece getting this part of the graph closer to the stripe. Essentially, the operation is similar to upper squeezing, only we squeeze symmetrically from both sides and leave the stripe between the two lines $y=\alpha x \pm C$ unchanged. See \autoref{fig:squeezing}.

\begin{figure} \centering
\begin{tikzpicture}[scale=0.5]
\tikzstyle{every node}=[shape=rectangle,fill=none,draw=none,minimum size=0cm,inner sep=2pt]
    \tikzstyle{every path}=[draw=black,line width = 1pt]

    \draw[->] (0,0)  to   (23.5,0);
    \draw[->] (0,0)  to   (0,11.5);

    \tikzstyle{every path}=[draw=black,line width = 0.5pt]    
\draw [gray!50] (0,0) grid (23,11);
\draw (0,0.8) -- (23,10.9);
\draw (1,0) -- (23,9.6);

    \tikzstyle{every path}=[draw=black,line width = 1pt]      
\draw (0,0)  --  (1,1) -- (2,2) -- (3,2) -- (4,2) -- (5,2) -- (6,2) -- (7,2) -- (8,2) -- (9,3) -- (10,3) -- (11,4) -- (12,5) -- (13,6) -- (14,6) -- (15,7) -- (16,8) -- (17,9) -- (18,10) -- (19,10) -- (20,10) -- (21,10) -- (22,10) -- (23,10);

\draw[dotted] (1,1) -- (2,1) -- (3,2)
(7,2) -- (8,3) -- (9,3) -- (10,4) -- (11,4)
(17,9) -- (18,9) -- (19,10);

 \node[fill=white] at (0.5,-0.5){$1$};
  \node[fill=white] at (1.5,-0.5){$1$};
   \node[fill=white] at (2.5,-0.5){$0$};
    \node[fill=white] at (3.5,-0.5){$0$};

 \node[fill=white] at (4.5,-0.5){$0$};
  \node[fill=white] at (5.5,-0.5){$0$};
   \node[fill=white] at (6.5,-0.5){$0$};
    \node[fill=white] at (7.5,-0.5){$0$};
    
 \node[fill=white] at (8.5,-0.5){$1$};
  \node[fill=white] at (9.5,-0.5){$0$};
   \node[fill=white] at (10.5,-0.5){$1$};
    \node[fill=white] at (11.5,-0.5){$1$};
    
 \node[fill=white] at (12.5,-0.5){$1$};
  \node[fill=white] at (13.5,-0.5){$0$};
   \node[fill=white] at (14.5,-0.5){$1$};
    \node[fill=white] at (15.5,-0.5){$1$};
    
 \node[fill=white] at (16.5,-0.5){$1$};
  \node[fill=white] at (17.5,-0.5){$1$};
   \node[fill=white] at (18.5,-0.5){$0$};
    \node[fill=white] at (19.5,-0.5){$0$};

  \node[fill=white] at (20.5,-0.5){$0$};
   \node[fill=white] at (21.5,-0.5){$0$};
    \node[fill=white] at (22.5,-0.5){$0$};
    
\draw (6,-2)  --  (7,-2);
\draw[densely dotted] (10,-2)  --  (11,-2);

 \node[fill=white] at (7.5,-2){$\infw{u}$};
 \node[fill=white] at (13,-2){$\infw{u}'=s_C(\infw{u})$};    
\end{tikzpicture}
\caption{Squeezing.}\label{fig:squeezing}
\end{figure}
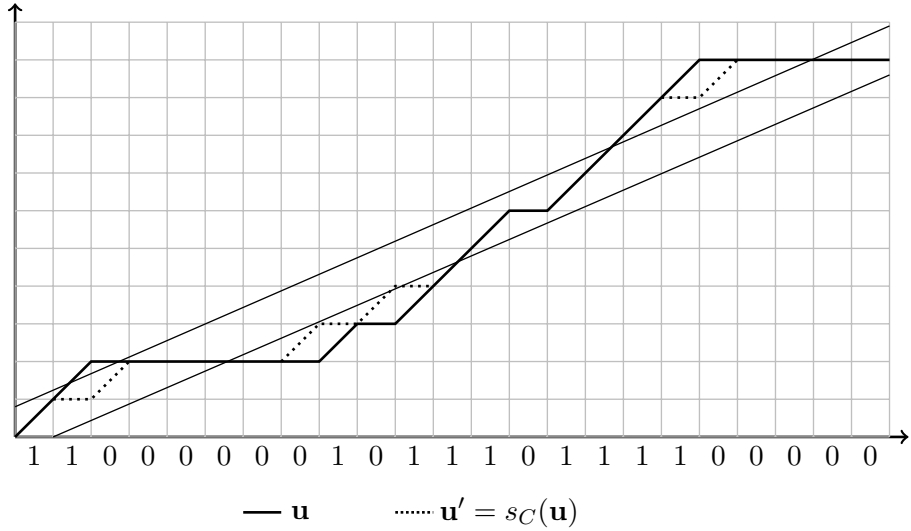

Clearly, similarly to upper squeezing (\autoref{claim:SqueezeFrequency}), the operation of squeezing does not change frequency.

The following Lemma generalizes \autoref{lem:uppersqueezing} for squeezing from both sides:

\begin{lemma}\label{lem:squeezing} Let $\infw{u}' = s_C(\infw{u})$, where $\infw{u}$ is as in \autoref{prop:unbal}. Then $\infw{u}' \in \abclsr{\infw{u}}$.
\end{lemma}

\begin{proof} The proof is similar to the proof of  \autoref{lem:uppersqueezing}, although here we have to consider more cases (for convenience of the reader, we repeat some arguments to have a complete proof here).

Assume the converse. Then, due to the \hyperref[lem:corridorLemma]{Corridor Lemma} there exists a factor $u'_i\cdots u'_{j-1}$ such that for each $k$ we have $|u'_i\cdots u'_{j-1}|_1>|u_k\cdots u_{k+j-i-1}|_1$---\hypertarget{assumption}{Assumption (*)} (the case of $<$ is symmetric).

First we remark that the switches inside the factor (at positions $i+1,\dots,j-1$) do not change the Parikh vector of the factor. So, to change Parikh vector, we must have a switch at position $i$ or/and $j$. 

Secondly, note that we have a switch at position $\ell$ if and only if $g_{\infw{u}}(\ell)\neq g_{\infw{u}'}(\ell)$.

\begin{enumerate}[leftmargin=*]
    \item Switch at $i$ and not in $j$. 

If $g_{\infw{u}}(i)<\alpha i - C$, then $u'_i=0$, $u_i=1$ and considering $k=i$, we get $|u_i\cdots u_{j-1}|_1>|u'_i\cdots u'_{j-1}|_1$, which is not possible by our \hyperlink{assumption}{Assumption (*)}.

The case $g_{\infw{u}}(i)>\alpha i + C$ is the same as Case 1 from the proof of \autoref{lem:uppersqueezing}: we then have  $u'_i=1$, $u_i=0$.
The only possibility is that $|u'_i\cdots u'_{j-1}|_1>\lceil \alpha(j-i)\rceil$ (since otherwise the factor $u'_i\cdots u'_{j-1}$ is in the set of abelian factors of $\infw{u}$). This in turn means that $g_{\infw{u}'}(j) > \alpha j +C$ (due to frequency). By the conditions of Case 1 we have that $g_{\infw{u}'}(j)=g_{\infw{u}}(j)$, which means that $u_{j-1}u_j\neq 10$. If $u_{j-1}=0$, then we taking $k=i-1$ we get an abelian equivalent factor in $\infw{u}$ (see \autoref{fig:C-balA}). If $u_j=1$, then we can take $k=i+1$  (see \autoref{fig:C-balB}). 

\item Switches at both $i$ and $j$. 

If $g_{\infw{u}}(i)<\alpha i - C$ and $g_{\infw{u}}(j)<\alpha j - C$ or $g_{\infw{u}}(i)>\alpha i + C$ and $g_{\infw{u}}(j)>\alpha j + C$, then the Parikh vector does not change (we can take $k=i$).

If  $g_{\infw{u}}(i)<\alpha i - C$ and $g_{\infw{u}}(j)>\alpha j + C$ or $g_{\infw{u}}(i)>\alpha i + C$ and $g_{\infw{u}}(j)<\alpha j - C$, then the new Parikh vector gets ``closer'' to the frequency, so it is evident that it belongs to the set of Parikh vectors of $\infw{u}$.

\item switch at $j$ and not in $i$. 
The case is symmetric to Case 1. \qedhere
\end{enumerate}
\end{proof}



\begin{proof}[Proof of \autoref{prop:unbal}]
We may assume without loss of generality that all $1$s are isolated in $\infw{x}$ by \autoref{cor:AbelianSsNonBalancedContainsIsolated1}. Further, we may assume that the
graph of $\infw{x}$ intersects the line $y = \alpha x$ infinitely often by \autoref{lem:line}. 
Consider the $C$-squeezing operation on $\infw{x}$: since all $1$s are isolated, the $C$-squeezing operation acts like the shift operation $\sigma$ on the parts that are outside the strip. Since the graph must contain arbitrarily long parts outside the stripe,
each factor of $\infw{x}$ is contained in the $s_C(\infw{x})$.

Consider a factor $w$ of $\infw{x}$ that is not $2C$-balanced. It occurs within bounded gaps
in $\infw{x}$, and all iterations of $s_{C}$ on $\infw{x}$. Notice though that the gaps could
grow in length. Now iterating the $C$-squeezing operation, we get arbitrarily long prefixes that are $2C$-balanced. So we get longer and longer gaps. This means that the longest $2C$-balanced factors grow in length, when iterating $s_C$. Since each of them is 
contained
in $\abclsr{\infw{x}}$, the claim follows.
\end{proof}






\section{Conclusions and open problems} \label{sec:conclusions}

In this paper, we 
studied a notion of abelian closures of infinite binary words. An interesting open question is to characterize words for which
$\abclsr{\infw{x}} = \soc{\infw x}$. Among uniformly recurrent binary words, this property
gives a characterization of Sturmian words, but the characterization does not extend to usual
generalizations of Sturmian words over non-binary alphabets: neither for balanced words, nor
for words of minimal complexity, nor for Arnoux-Rauzy words \cite{KarhumakiPW:on_abelian_subshifts,DBLP:conf/cwords/Puzynina19,HejdaSteinerZamboni15}. 

\begin{openproblem}\label{op:equals}
Find a characterization of the property 
$\soc{\infw{x}} = \abclsr{\infw{x}}$ for nonbinary
alphabets.
\end{openproblem}

A modification of
this question is to characterize words for which $\abclsr{\infw x}$ contains exactly one
minimal subshift.

Another question to study concerns abelian closures of binary
words and \autoref{thm:binary}. We showed that the abelian closures of uniformly recurrent
binary words contain infinitely 
many minimal subshifts. We also showed that in fact there are uncountably many minimal subshifts unless the frequency exists and it is irrational. The proof is quite technical and consists of four parts relating to the cases of rational frequency, no letter frequencies, balanced and unbalanced words with irrational letter frequencies, and the proofs of all the parts rely on different methods. It would be interesting to try to find a proof treating all cases at once and giving a stronger result of uncountably many minimal subshifts in all the cases:

\begin{openproblem}
Find a shorter proof of \autoref{thm:binary}. Does it hold if we substitute ``infinitely many'' by ``uncountably many" in the case of irrational frequencies?
\end{openproblem}

We remark that we were able to prove the problem for a wide class of such words (see
\autoref{prop:wideStripUncountablyMany}). In fact, by Propositions \ref{prop:NoFrequency}, 
\ref{prop:rationalFreqUncountable}, and \ref{prop:wideStripUncountablyMany},
if there is a uniformly recurrent binary word $\infw{x}$ for which
$\abclsr{\infw{x}}$ contains infinitely, but only countably
many minimal subshifts, then it must have irrational letter
frequency $\alpha$ and, further, $\min_{v \in \lang[n]{\infw{x}}}|v|_1 = \lfloor \alpha n\rfloor$ or $\max_{v \in \lang[n]{\infw{x}}}|v|_1 = \lceil \alpha n\rceil$ for infinitely
many $n$.

A quantitative version of the above question would be
"Does the abelian closure of a uniformly recurrent non-Sturmian aperiodic binary word have \emph{positive entropy}?" See, e.g., \cite{LindMarcus95} for a definition of entropy.
The proof of \autoref{prop:rationalFreqUncountable} implies that the entropy is positive for words with rational letter frequencies. This can be translated to the case of no frequencies also.

\section*{Acknowledgements}
	Svetlana Puzynina is partially supported by Russian Foundation of Basic Research
(grant 20-01-00488) and by the Foundation for the Advancement of Theoretical Physics and Mathematics
``BASIS''. Part of the research performed while Markus Whiteland was at the Department of Mathematics and Statistics, University of Turku, Finland. Markus Whiteland would like to thank Joonatan Jalonen for
interesting discussions on the topic.

\bibliographystyle{abbrvnat}
\biboptions{sort&compress}
\bibliography{bibliography}

\begin{thebibliography}{30}
\providecommand{\natexlab}[1]{#1}
\providecommand{\url}[1]{\texttt{#1}}
\expandafter\ifx\csname urlstyle\endcsname\relax
  \providecommand{\doi}[1]{doi: #1}\else
  \providecommand{\doi}{doi: \begingroup \urlstyle{rm}\Url}\fi

\bibitem[Avgustinovich and Puzynina(2016)]{AP16}
S.~V. Avgustinovich and S.~Puzynina.
\newblock Weak abelian periodicity of infinite words.
\newblock \emph{Theory of Computer Systems}, 59:\penalty0 161--179, 2016.
\newblock \doi{10.1007/s00224-015-9629-1}.

\bibitem[Blanchet{-}Sadri et~al.(2014)Blanchet{-}Sadri, Fox, and
  Rampersad]{DBLP:journals/aam/Blanchet-SadriF14}
F.~Blanchet{-}Sadri, N.~Fox, and N.~Rampersad.
\newblock On the asymptotic abelian complexity of morphic words.
\newblock \emph{Advances in Applied Mathematics}, 61:\penalty0 46--84, 2014.
\newblock \doi{10.1016/j.aam.2014.08.005}.

\bibitem[Cassaigne et~al.(2011)Cassaigne, Richomme, Saari, and
  Zamboni]{DBLP:journals/ijfcs/CassaigneRSZ11}
J.~Cassaigne, G.~Richomme, K.~Saari, and L.~Q. Zamboni.
\newblock Avoiding abelian powers in binary words with bounded abelian
  complexity.
\newblock \emph{International Journal of Foundations of Computer Science},
  22\penalty0 (4):\penalty0 905--920, 2011.
\newblock \doi{10.1142/S0129054111008489}.

\bibitem[Constantinescu and Ilie(2006)]{DBLP:journals/eatcs/ConstantinescuI06}
S.~Constantinescu and L.~Ilie.
\newblock Fine and wilf's theorem for abelian periods.
\newblock \emph{{EATCS} Bull.}, 89:\penalty0 167--170, 2006.

\bibitem[Coven and Hedlund(1973)]{DBLP:journals/mst/CovenH73}
E.~M. Coven and G.~A. Hedlund.
\newblock {Sequences with Minimal Block Growth}.
\newblock \emph{Math. Syst. Theory}, 7\penalty0 (2):\penalty0 138--153, 1973.
\newblock \doi{10.1007/BF01762232}.

\bibitem[de~Luca(1996)]{deLuca1996:on_standard_sturmian_moprhisms}
A.~de~Luca.
\newblock On standard {S}turmian morphisms.
\newblock In F.~Meyer and B.~Monien, editors, \emph{Automata, Languages and
  Programming}, pages 403--415, Berlin, Heidelberg, 1996. Springer Berlin
  Heidelberg.
\newblock \doi{10.1016/S0304-3975(96)00054-0}.

\bibitem[{de
  Luca}(1997)]{deLuca1996:sturmian_words_structure_combinatorics_arithmetics}
A.~{de Luca}.
\newblock {S}turmian words: structure, combinatorics, and their arithmetics.
\newblock \emph{Theoretical Computer Science}, 183\penalty0 (1):\penalty0
  45--82, 1997.
\newblock \doi{10.1016/S0304-3975(96)00310-6}.

\bibitem[Durand(1998)]{durand}
F.~Durand.
\newblock A characterization of substitutive sequences using return words.
\newblock \emph{Discrete Mathematics}, 179:\penalty0 89--101, 1998.
\newblock \doi{10.1016/S0012-365X(97)00029-0}.

\bibitem[Fici et~al.(2017)Fici, Mignosi, and
  Shallit]{DBLP:journals/tcs/FiciMS17}
G.~Fici, F.~Mignosi, and J.~O. Shallit.
\newblock Abelian-square-rich words.
\newblock \emph{Theor. Comput. Sci.}, 684:\penalty0 29--42, 2017.
\newblock \doi{10.1016/j.tcs.2017.02.012}.

\bibitem[Hejda et~al.(2015)Hejda, Steiner, and Zamboni]{HejdaSteinerZamboni15}
T.~Hejda, W.~Steiner, and L.~Q. Zamboni.
\newblock What is the {A}belianization of the {T}ribonacci shift?, 2015.
\newblock Workshop on Automatic Sequences, Li{\' e}ge, May 2015.

\bibitem[Holton and Zamboni(1998)]{hz}
C.~Holton and L.~Q. Zamboni.
\newblock Geometric realizations of substitutions.
\newblock \emph{Bulletin de la Soci\'{e}t\'{e} Math\'{e}matique de France},
  126:\penalty0 149--179, 1998.
\newblock \doi{10.24033/bsmf.2324}.

\bibitem[Hubert(2000)]{Hubert00}
P.~Hubert.
\newblock Suites {\'{e}}quilibr{\'{e}}es.
\newblock \emph{Theor. Comput. Sci.}, 242\penalty0 (1-2):\penalty0 91--108,
  2000.
\newblock \doi{10.1016/S0304-3975(98)00202-3}.

\bibitem[Karhum{\"{a}}ki et~al.(2018)Karhum{\"{a}}ki, Puzynina, and
  Whiteland]{KarhumakiPW:on_abelian_subshifts}
J.~Karhum{\"{a}}ki, S.~Puzynina, and M.~A. Whiteland.
\newblock On abelian subshifts.
\newblock In M.~Hoshi and S.~Seki, editors, \emph{DLT 2018}, volume 11088 of
  \emph{Lecture Notes in Computer Science}, pages 453--464. Springer, 2018.
\newblock \doi{10.1007/978-3-319-98654-8\_37}.

\bibitem[Ker{\"{a}}nen(1992)]{DBLP:conf/icalp/Keranen92}
V.~Ker{\"{a}}nen.
\newblock Abelian squares are avoidable on 4 letters.
\newblock In \emph{ICALP92}, volume 623 of \emph{Lecture Notes in Computer
  Science}, pages 41--52. Springer, 1992.
\newblock \doi{10.1007/3-540-55719-9\_62}.

\bibitem[Lind and Marcus(1995)]{LindMarcus95}
D.~Lind and B.~Marcus.
\newblock \emph{An Introduction to Symbolic Dynamics and Coding}.
\newblock Camb. Univ. Press, New York, NY, USA, 1995.
\newblock ISBN 0{-}521{-}55900{-}6.

\bibitem[Lothaire(1983)]{lothaire1983combinatorics}
M.~Lothaire.
\newblock \emph{Combinatorics on {Words}}, volume~17 of \emph{Encycl. Math.
  Appl.}
\newblock Addison-Wesley, 1983.
\newblock ISBN 978-0-201-13516-9.

\bibitem[Lothaire(2002)]{MR1905123}
M.~Lothaire.
\newblock \emph{Algebraic combinatorics on words}, volume~90 of \emph{Encycl.
  Math. Appl.}
\newblock Cambridge University Press, 2002.
\newblock ISBN 0-521-81220-8.
\newblock \doi{10.1017/CBO9781107326019}.

\bibitem[Madill and Rampersad(2013)]{DBLP:journals/dm/MadillR13}
B.~Madill and N.~Rampersad.
\newblock The abelian complexity of the paperfolding word.
\newblock \emph{Discret. Math.}, 313\penalty0 (7):\penalty0 831--838, 2013.
\newblock \doi{10.1016/j.disc.2013.01.005}.

\bibitem[Mignosi and
  S\'e\'ebold(1993)]{MignosiSeebold1993:morphismes_sturmiens_et_regles_de_rauzy}
F.~Mignosi and P.~S\'e\'ebold.
\newblock Morphismes sturmiens et r\`egles de {R}auzy.
\newblock \emph{Journal de th\'eorie des nombres de Bordeaux}, 5\penalty0
  (2):\penalty0 221--233, 1993.
\newblock URL \url{http://www.numdam.org/item/JTNB_1993__5_2_221_0}.

\bibitem[Morse and Hedlund(1940)]{Morse10.2307/2371431}
M.~Morse and G.~A. Hedlund.
\newblock {Symbolic Dynamics II. Sturmian Trajectories}.
\newblock \emph{Am. J. Math.}, 62:\penalty0 1--42, 1940.

\bibitem[Peltom{\"{a}}ki and Whiteland(2019)]{DBLP:conf/cwords/PeltomakiW19}
J.~Peltom{\"{a}}ki and M.~A. Whiteland.
\newblock Every nonnegative real number is an abelian critical exponent.
\newblock In \emph{WORDS 2019}, volume 11682 of \emph{Lecture Notes in Computer
  Science}, pages 275--285. Springer, 2019.
\newblock \doi{10.1007/978-3-030-28796-2\_22}.

\bibitem[Peltom{\"{a}}ki and
  Whiteland(2020{\natexlab{a}})]{PeltomakiW2020:AllGrowthRatesOfAbelianExponents}
J.~Peltom{\"{a}}ki and M.~A. Whiteland.
\newblock All growth rates of abelian exponents are attained by infinite binary
  words.
\newblock In J.~Esparza and D.~Kr{\'{a}}l', editors, \emph{{MFCS} 2020}, volume
  170 of \emph{LIPIcs}, pages 79:1--79:10. Schloss Dagstuhl - Leibniz-Zentrum
  f{\"{u}}r Informatik, 2020{\natexlab{a}}.
\newblock \doi{10.4230/LIPIcs.MFCS.2020.79}.

\bibitem[Peltom{\"{a}}ki and
  Whiteland(2020{\natexlab{b}})]{PeltomakiW2020:AvoidingAbelianPowersCyclically}
J.~Peltom{\"{a}}ki and M.~A. Whiteland.
\newblock Avoiding abelian powers cyclically.
\newblock \emph{Advances in Applied Mathematics}, 121:\penalty0 102095,
  2020{\natexlab{b}}.
\newblock \doi{10.1016/j.aam.2020.102095}.

\bibitem[Puzynina(2019)]{DBLP:conf/cwords/Puzynina19}
S.~Puzynina.
\newblock Abelian properties of words.
\newblock In \emph{WORDS 2019}, volume 11682 of \emph{Lecture Notes in Computer
  Science}, pages 28--45. Springer, 2019.
\newblock \doi{10.1007/978-3-030-28796-2\_2}.

\bibitem[Puzynina and Zamboni(2013)]{PUZYNINA2013390}
S.~Puzynina and L.~Q. Zamboni.
\newblock Abelian returns in {S}turmian words.
\newblock \emph{J. Comb. Theory Ser. A}, 120\penalty0 (2):\penalty0 390--408,
  2013.
\newblock \doi{10.1016/j.jcta.2012.09.002}.

\bibitem[Rao and Rosenfeld(2018)]{DBLP:journals/siamdm/RaoR18}
M.~Rao and M.~Rosenfeld.
\newblock Avoiding two consecutive blocks of same size and same sum over
  {$\mathbb{Z}^{2}$}.
\newblock \emph{{SIAM} J. Discret. Math.}, 32\penalty0 (4):\penalty0
  2381--2397, 2018.
\newblock \doi{10.1137/17M1149377}.

\bibitem[Richomme et~al.(2011)Richomme, Saari, and
  Zamboni]{DBLP:journals/jlms/RichommeSZ11}
G.~Richomme, K.~Saari, and L.~Q. Zamboni.
\newblock Abelian complexity of minimal subshifts.
\newblock \emph{J. Lond. Math. Soc.}, 83\penalty0 (1):\penalty0 79--95, 2011.
\newblock \doi{10.1112/jlms/jdq063}.

\bibitem[Saarela(2009)]{DBLP:journals/jalc/Saarela09}
A.~Saarela.
\newblock Ultimately constant abelian complexity of infinite words.
\newblock \emph{J. Autom. Lang. Comb.}, 14\penalty0 (3/4):\penalty0 255--258,
  2009.
\newblock \doi{10.25596/jalc-2009-255}.

\bibitem[Wolfram(1983)]{Wolfram1983:statistical_mechanics_of_cellular_automata}
S.~Wolfram.
\newblock Statistical mechanics of cellular automata.
\newblock \emph{Reviews of Modern Physics}, 55:\penalty0 601--644, 1983.
\newblock \doi{10.1103/RevModPhys.55.601}.

\bibitem[Zamboni(2018)]{ZamboniPersonal}
L.~Q. Zamboni.
\newblock Personal communication, 2018.

\end{thebibliography}

\end{document}